\documentclass[twoside,11pt]{article}
\usepackage{amsmath,amsfonts,amssymb,booktabs,float,graphicx}
\usepackage{algorithm, algpseudocode,amsthm}
\usepackage[utf8]{inputenc}
\usepackage[left=2cm,right=2cm,top=2cm,bottom=2cm]{geometry}
		
\newcommand{\soln}{u}
\newcommand{\dsoln}{\dot u}

\newcommand{\solnpjk}[3]{u_{#2}^{#1,\{#3\}}}

\newcommand{\fsolnp}[1]{U^{#1}}
\newcommand{\fsolnj}[1]{U_{#1}}

\newcommand{\fsolnpk}[2]{U^{#1,\{#2\}}}

\newcommand{\fsolnpjk}[3]{U_{#2}^{#1,\{#3\}}}

\newcommand{\fsolnjk}[2]{U_{#1}^{\{#2\}}}

\newcommand{\fsolnCG}{U_{CG}}
\newcommand{\dfsolnCG}{\dot{U}_{CG}}
\newcommand{\dfsolnDG}{\dot{U}_{DG}}

\newcommand{\fsolnDGp}[1]{U_{DG}^{#1}}

\newcommand{\fsolnDG}{U_{DG}}

\newcommand{\Var}{\mathrm{Var}}
\newcommand{\fsolnVAR}{U_{\Var}}

\newcommand{\ferrVAR}{e_{\Var}}

\newcommand{\VARrespuphi}[3]{\mathcal{R}_{\Var}^{#1}(#2,#3)}

\newcommand{\fsolnkl}[2]{U^{\{#1,#2\}}}
\newcommand{\fsolnpkl}[3]{U^{#1,\{#2,#3\}}}

\newcommand{\fsolnpjkl}[4]{U_{#2}^{#1,\{#3, #4\}}}

\newcommand{\fsolnk}[1]{U^{\{#1\}}}

\newcommand{\fsolpnkidd}[4]{\widetilde{U}_{#2,#3}^{#1,\{#4\}} }
\newcommand{\fsolnikdd}[2]{\widetilde{U}_{n,#1}^{\{#2\}} }

\newcommand{\csoln}{\widehat{U}}

\newcommand{\ferrDG}{e_{DG}}

\newcommand{\csolnp}[1]{\widehat{U}^{#1}}

\newcommand{\csolnpk}[2]{\widehat{U}^{#1,\{#2\}}}

\newcommand{\csolnk}[1]{\widehat{U}^{\{#1\}}}
\newcommand{\csolnj}[1]{\widehat{U}_{#1}}

\newcommand{\ferrp}[1]{e^{#1}}
\newcommand{\ferrk}[1]{e^{\{#1\}}}
\newcommand{\ferrkl}[2]{e^{\{#1,#2\}}}
\newcommand{\ferrpk}[2]{e^{#1, \{#2\}  }}

\newcommand{\cerrp}[1]{\widehat{e}^{#1}}
\newcommand{\cerrk}[1]{\widehat{e}^{\{#1\}}}
\newcommand{\cerrpk}[2]{\widehat{e}^{#1, \{#2\}  }}

\newcommand{\adjoint}{\phi}
\newcommand{\fadj}{\phi}
\newcommand{\fadjp}[1]{\phi^{#1}}
\newcommand{\dfadjp}[1]{\dot{\phi}^{#1}}

\newcommand{\cadj}{\widehat{\phi}}

\newcommand{\adjdd}[2]{\phi_{{#2}}^{#1}}
\newcommand{\adjddadd}[4]{\phi_{{#2},#4}^{#1,[#3]}}
\newcommand{\ddpsi}[3]{\fadjp{#1}(t_{#2})|_{\Omega_{#3}}}

\newcommand{\auxadjp}[1]{\widehat{\phi}_{\rm aux}^{#1}}
\newcommand{\dauxadjp}[1]{\dot{\widehat{\phi}}_{\rm aux}^{#1}}

\newcommand{\frespuphi}[3]{\mathcal{R}^{#1}(#2,#3)}

\newcommand{\frespnarg}[4]{\mathcal{R}^{#1}_{#2}(#3, #4)}
\newcommand{\crespnarg}[4]{\widehat{\mathcal{R}}^{#1}_{#2}(#3, #4)}

\newcommand{\crespuphi}[3]{\widehat{\mathcal{R}}^{#1}(#2,#3)}

\newcommand{\Nftime}{N_t}
\newcommand{\Nftimep}[1]{N^{#1}_t}
\newcommand{\Nctime}{\widehat{N}_t}
\newcommand{\Nctimep}[1]{\widehat{N}^{#1}_t}

\newcommand{\Qftime}{q_t}
\newcommand{\Qctime}{\widehat{q}_t}

\newcommand{\Nfspace}{N_s}
\newcommand{\Ncspace}{\widehat{N}_s}

\newcommand{\Qfspace}{q_s}
\newcommand{\Qcspace}{\widehat{q}_s}

\newcommand{\Vspace}[1]{{V}_s^{#1}}
\newcommand{\Vspacetime}[3]{{W}_{p,#1}^{#2,#3}}

\newcommand{\Vspacedd}[2]{V_{{#2}}^{#1}}
\newcommand{\Vspaceddbc}[3]{V_{{#2},{#3}}^{#1}}

\newcommand{\NParSub}{P_t}
\newcommand{\NDDSub}{P_s}

\newcommand{\NParIte}{K_t}
\newcommand{\NDDIte}{K_s}

\newcommand{\ftimepj}[2]{t^{#1}_{#2}}
\newcommand{\Intfpj}[2]{\mathcal{I}^{#1}_{#2}}
\newcommand{\Taufp}[1]{\mathcal{T}^{#1}}
\newcommand{\fdTpj}[2]{\Delta t^{#1}_{#2}}
\newcommand{\cdTpj}[2]{\Delta \hat{t}^{#1}_{#2}}

\newcommand{\ctimepj}[2]{\widehat{t}^{#1}_{#2}}
\newcommand{\Intcpj}[2]{\widehat{\mathcal{I}}^{#1}_{#2}}
\newcommand{\Taucp}[1]{\widehat{\mathcal{T}}^{#1}}

\newcommand{\fsolvernoargs}{F^p}
\newcommand{\fsolver}[2]{F^p[#1] (#2)}
\newcommand{\finv}[3]{\langle #1 \rangle_{ {\mathcal{I}}^{{#2}}_{#3} }  }

\newcommand{\csolvernoargs}{\widehat{G}}
\newcommand{\csolver}[2]{\widehat{G}^p[#1] (#2)}
\newcommand{\cinv}[3]{\langle #1 \rangle_{ \widehat{\mathcal{I}}^{#2}_{#3} }}

\newcommand{\fterror}{\mathcal{D}}
\newcommand{\auxerror}{\mathcal{A}}
\newcommand{\cterror}{\mathcal{C}}
\newcommand{\Kterror}{\mathcal{K}}

\newcommand{\fterrft}{\mathcal{D}_t}
\newcommand{\fterrNs}{\mathcal{D}_s}
\newcommand{\fterrKs}{\mathcal{D}_k}
\newcommand{\fterrftp}[1]{\mathcal{D}_t^{#1}}
\newcommand{\fterrNsp}[1]{\mathcal{D}_s^{#1}}
\newcommand{\fterrKsp}[1]{\mathcal{D}_k^{#1}}
\newcommand{\fterrNspn}[2]{\mathcal{D}_{s,#2}^{#1}}
\newcommand{\fterrKspn}[2]{\mathcal{D}_{k,#2}^{#1}}

\newcommand{\corr}[2]{C^{#1,\{#2\}}}

\newcommand{\reffactime}{r}

\newcommand{\lspace}[1]{l^{#1}}

\newtheorem{thm}{Theorem}
\newtheorem{lem}{Lemma}

\title{A posteriori error analysis for a space-time parallel discretization of parabolic partial differential equations}

\usepackage{fancyhdr}
\newcommand\shorttitle{Error analysis for space-time parallel discretization}

\usepackage{authblk}
\usepackage{hyperref}

\author[1]{Jehanzeb H. Chaudhry}
\author[2]{Don Estep}
\author[3]{Simon Tavener}
\affil[1]{Department of Mathematics and Statistics, University of New Mexico.   Email: \url{jehanzeb@unm.edu}}
\affil[2]{Department of Statistics and Actuarial Science, Simon Fraser University.  Email: \url{tavener@math.colostate.edu}}
\affil[3]{Department of Mathematics, Colorado State University. Email: \url{donald_estep@sfu.ca}}

\providecommand{\keywords}[1]{\textbf{\textit{Keywords 	---}} #1}

\begin{document}
\maketitle{}

\begin{abstract}
We construct a space-time parallel method for solving parabolic partial differential equations by coupling the Parareal algorithm in time with overlapping domain decomposition in space. The goal is to obtain a discretization consisting of ``local'' problems that can be solved on parallel computers efficiently. However, this introduces significant sources of error that must be evaluated. Reformulating the original Parareal algorithm as a variational method and implementing a finite element discretization in space enables an adjoint-based \emph{a posteriori} error analysis to be performed. Through an appropriate choice of adjoint problems and residuals the error analysis distinguishes between errors arising due to the temporal and spatial discretizations, as well as between the errors arising due to incomplete Parareal iterations and incomplete iterations of the domain decomposition solver. We first develop an error analysis for the Parareal method applied to  parabolic partial differential equations, and then refine this analysis to the case where the associated spatial problems are solved using overlapping domain decomposition. These constitute our Time Parallel Algorithm (TPA) and Space-Time Parallel Algorithm (STPA) respectively.  Numerical experiments demonstrate the accuracy of the estimator for both algorithms and the iterations between distinct components of the error. 
\end{abstract}

\keywords{A posteriori error analysis, adjoint based error estimation, parabolic partial differential equations, Parareal, domain decomposition, Schwarz algorithms, time-parallel.}

\section{Introduction}
\label{sec:introduction}

Parallel computing approaches for solving complex problems modeled by partial differential equations have become increasingly attractive over the last decade, as Moore's Law has continued, but transitioned from fabricating  more transistors on a chip to constructing  more cores. Spatial parallelization of partial differential equations is well established \cite{chan1994domain, dolean2015introduction,  lions1990schwarz, lions1988schwarz, xu1992iterative}, but the improvement available through spatial parallelization alone typically decays as the number of processors increases for any fixed-size problem (the so-called strong scaling behavior), see e.g., \cite{falgout2017multigrid}. Parallel in time methods provide an additional way to effectively utilize modern high performance computer architectures \cite{emmett2014efficient, falgout2014parallel, gander201550, gotschel2019efficient, howse2019parallel, maday2002parareal, ong2020applications}. Unfortunately parallel-in-time methods also introduce significant additional sources of error which must be estimated for the purposes of validation when seeking to construct adaptive algorithms, or when performing uncertainty quantification.

We consider a PDE system of the form: Find $\soln(x,t) \in L^2(0,T; H^1_0(\Omega))$  such that
\begin{equation}
\label{eq:pde}
\begin{aligned}
	(\dsoln, v) + a(\soln, v) &= l(v),  \quad  \forall v \in H^1_0(\Omega) \text{ and } t \in (0,T], \\
	   \soln(x,0) &= u_0(x),
\end{aligned}
\end{equation}
where $\dsoln = \frac{\partial \soln}{\partial t}$, $a(\cdot, \cdot)$ is a coercive, positive definite bilinear form, $l(\cdot)$ is a linear form,  $\Omega$ is a convex polygonal domain, and $u_0 \in H^1_0$ is the initial condition. Here $( \cdot, \cdot)$ denotes the $L_2(\Omega)$ inner product, so that $(w,v) = \int_{\Omega} w v \, {\rm d}x$.
 We assume that the forms $a$ and $l$ satisfy sufficient regularity assumptions so that a  weak solution $\soln(x,t) \in L^2(0,T; H^1_0(\Omega))$ exists and is unique \cite{lawrence2010evanspartial}. Further, we consider a quantity of interest (QoI) given by
\begin{equation}
\label{eq:qoi_def}
\mathcal{Q}(\soln) = \int_{\Omega} \psi(x) \soln(x,T) \, {\rm d}x
\end{equation}
where $\psi \in L^2(\Omega)$.

The aim of this work is to derive error estimates for the numerically computed value of $\mathcal{Q}$ when time-parallel and spatial domain decomposition techniques are employed for the approximation of $\soln(x,t)$. In particular, we consider two algorithms which we call the Time Parallel Algorithm (TPA) and the Space-Time Parallel Algorithm (STPA). The TPA implements the Parareal algorithm\cite{Emmett2012, GV07, MH08, Maday08, Maday2002387},  a time-parallel method to solve \eqref{eq:pde}, while the STPA also employs a parallel domain decomposition strategy in space \cite{Keyes:1995:DBP, Tarek2008, Smith:1996:DPM, Toselli:2004:DDM} in addition to the time-parallelism of the Parareal method.

The analysis presented here extends earlier work on the \emph{a posteriori} error analysis of the Parareal algorithm in \cite{chaudhry2016posteriori}, and incorporates the work of \cite{chaudhry2019posteriori} on the \emph{a posteriori} error analysis of overlapping Schwarz domain decomposition algorithms to provide an estimate of the spatial discretization error. One significant difference between the earlier work 
and the current work 
is that we analyze the effects of using an iterative method to solve the discrete equations arising from an implicit time integration scheme.
The analysis builds on early work of adjoint-based \emph{a posteriori} error analysis \cite{AO2000,   beckerrannacher, eriksson1995introduction, Giles:Suli:02} and more recent work addressing iterative and multi-rate methods \cite{CEGT13a, iternon, EGT2012}.


\subsection{Notation}
Our analysis requires consideration of both exact and discrete solutions. Analytical solutions are indicated using a lower case letter, discrete solutions with an upper case letter. Solutions of the PDE will be designated $u$ and adjoint solutions $\phi$. A hat indicates a coarse scale entity. Errors are indicated using $e$ and residuals with $\mathcal{R}$. A dot indicates a partial differential in time. We use $N$ to indicate the number of finite elements, $Q$ to indicate the degree of interpolation, $P$ to indicate the number of temporal or spatial subdomains, and $K$ to indicate the number of iterations. A subscript $t$ for these symbols indicates a temporal quantity, a subscript $s$ indicates a spatial quantity.  A full list of variables and discretization choices appears in \ref{sec:appendix}.

For a function $w(x,t)$, we denote  $w(t):= w(\cdot,t)$, i.e., $w(t)$ refers to a function of space only for a fixed time $t$. We use the notation $(w,v)(t) := (w(t),v(t)) =  \int_\Omega w(x,t) v(x,t) \, {\rm d}x$, so for example the QoI~\eqref{eq:qoi_def} may be expressed as $\mathcal{Q}(\soln) = (\psi, \soln)(T)$. Finally, we let $\langle \cdot \rangle_I$ indicate integration over time interval $I$, i.e. $\langle \cdot \rangle_I = \int_I \, \cdot \, {\rm d}t$.

\subsection{Outline}
We begin the analysis from the perspective of Parareal integration in time, presented in \cite{chaudhry2016posteriori}. The Parareal approach utilizes both coarse and fine scale discretizations in time. Extending this analysis to PDEs requires corresponding coarse and fine discretizations in space. Since we implement a domain decomposition strategy in space, we  employ iterative solution methods in both time and space.


We introduce two algorithms in section \ref{sec:parareal_algorithms}: a Time Parallel Algorithm (TPA), and a Space-Time Parallel Algorithm (SPTA).  The TPA employs  the time-parallel algorithm, Parareal, for  the solution of PDEs, whereas the SPTA additionally utilizes additive Schwarz domain decomposition to achieve parallelization in space.
This section also contains precise details of the spatial and temporal discretizations. 
We recall  some essential mathematical results needed for the subsequent analysis for the two algorithms in section~\ref{sec:var_mthds}. We
adjust the previous \emph{a posteriori} error analysis for Parareal time integration to analyze the TPA in in section \ref{sec:temporal_errors}. 
We extend this analysis to account for the errors arising using additive Schwarz domain decomposition in section \ref{sec:spatial_errors}. Numerical results supporting the accuracy of the \emph{a posteriori} error estimates  are provided in section \ref{sec:numerical_results}.

\section{Discretizations: The ``Time Parallel Algorithm'' and the  ``Space-Time Parallel Algorithm''}
\label{sec:parareal_algorithms}

We consider two algorithms: (1) An algorithm which is parallel in time only, referred to as Time Parallel Algorithm (TPA), and (2) An algorithm which is parallel in both space and time, referred to as Space-Time Parallel  Algorithm (STPA). Both algorithms use the Parareal algorithm for time-parallelism. The spatial parallelization in STPA is achieved through the use of overlapping Schwarz domain decomposition method. The discretization also involves using the finite element method in space and the implicit Euler method in time. The implicit Euler method is appropriate for dissipative problems modeled by \eqref{eq:pde}. We describe these algorithms in detail below.

\subsection{Time Parallel Algorithm (TPA)}
\label{sec:TPA}
The Parareal algorithm is based on a pair of coarse and fine scale solvers, $\csolver{\alpha(x)}{t}$ and $\fsolver{\alpha(x)}{t}$ respectively. Let $0 = T_0 < \cdots < T_{p-1} < T_p < \cdots < T_P = T$ be a partition of the time domain. We view the times $T_p$ as ``synchronization'' times at which point the coarse and fine scale solutions may exchange information. We call the time interval $[T_{p-1},T_p]$ the temporal subdomain $p$, see Figure \ref{fig:subdomains_and_discretizations}. Note that the temporal subdomains are non-overlapping.

The numerical solver $\fsolver{\alpha(x)}{t}$  is more accurate than $\csolver{\alpha(x)}{t}$, e.g., employing a finer discretization or a higher order method. The super-script $p$ indicates that the solvers produce a space-time solution for the temporal subdomain $p$. That is, for $t \in (T_{p-1},T_p]$, $\csolver{\alpha(x)}{t}$ (resp. $\fsolver{\alpha(x)}{t}$) indicates the space-time coarse scale (resp. fine scale) solution at $t$, where $\alpha(x)$ denotes the initial conditions for the solver at $t = T_{p-1}$. 
Specific examples of the solvers are presented in \ref{sec:time_space_discretization}.



The standard version of the Parareal algorithm defines the solutions only at the times $T_p$ and is not amenable to adjoint-based \emph{a posteriori} analysis, which requires solutions to be in variational form. An equivalent, variational version of the Parareal algorithm, suitable for an adjoint-based \emph{a posteriori} error analysis, is provided in Algorithm \ref{alg:parareal_variational}. The standard Parareal algorithm is provided in \ref{sec:appendix_Parareal} and a proof of its equivalence with the variational version is given in \cite{chaudhry2016posteriori}. Here $\csolnpk{p}{k_t}(x,t) $ and $\fsolnpk{p}{k_t}(x,t)$ are the coarse and fine scale solutions for the variational Parareal algorithm at iteration $k_t$ on temporal subdomain $p$. The fine-scale Parareal solution after $k_t$ iterations is defined to be $\fsolnk{k_t} = \fsolnpk{p}{k_t}$ for $t \in (T_{p-1}, T_p]$. The initial condition approximation to $\soln_0$ is referred to as $\csoln_0$. Note that since $\csoln_0$ is represented in a finite dimensional space,  it is possible that $\soln_0 \neq \csoln_0$.

\begin{algorithm}[H]\caption{Variational form of the Parareal algorithm}
\label{alg:parareal_variational}
\begin{algorithmic}
\Procedure{VPAR}{$\NParSub, \NParIte, \csoln_0$ }
\State $\corr{p}{0} = 0, p=0,\ldots,\NParSub$ \Comment{Initialize corrections}
\For{$k_t = 1, \ldots, \NParIte$}
  \State $\csolnpk{0}{k_t}(x,0):= \csoln_0$ \Comment{Set initial conditions}
  \For{$p=1, \ldots, \NParSub$}
    \State $\csolnpk{p}{k_t}(x,t) = \csolver{\csolnpk{p-1}{k_t}(x,T_{p-1})  + \corr{p}{k_t-1}(x)}{t}$
           \Comment{Serial computation}
    \State $\fsolnpk{p}{k_t}(x,t) = \fsolver{\csolnpk{p}{k_t}(x,T_{p-1})}{t}$
           \Comment{Parallel computation}
    \State $\corr{p}{k_t}(x) = \fsolnpk{p}{k_t}(x,T_p) - \csolver{\csolnpk{p}{k_t}(x,T_{p-1})}{T_p}$
           \Comment{Update corrections}
  \EndFor
\EndFor
\State \Return $\fsolnpk{p}{K_t}(x,t)$
\EndProcedure
\end{algorithmic}
\end{algorithm}


\subsubsection{Temporal and spatial discretizations of the fine scale and coarse scale solvers}
\label{sec:time_space_discretization}

We consider the implicit Euler method for time discretization and finite elements for the spatial discretization to define coarse and fine scale solutions on $[T_{p-1}, T_p]$. We omit the superscript $p$ in this section when the temporal subdomain is clear.

We discretize the \textbf{spatial} domain $\Omega$ into a quasi-uniform triangulation $\mathcal{T}_{h}$, where $h$ denotes the maximum diameter of the elements. This triangulation is chosen so the union of the elements of $\mathcal{T}_{h}$ is $\Omega$ and the intersection of any two elements is either a common edge, node, or is empty. Let $\Vspace{q}$ be the standard space of continuous piecewise polynomials of degree $q$ on $\mathcal{T}_{h}$, such that if $U \in \Vspace{q}$ then $U = 0$ on $\partial \Omega$. In particular, we consider a coarse space $\Vspace{\Qcspace}$ and fine space $\Vspace{\Qfspace}$ such that $\Qcspace < \Qfspace$.

On each \textbf{time} subdomain $p$, the coarse temporal discretization, uses $\Nctimep{p}$ time steps while the fine temporal discretization uses $\Nftimep{p}$ time steps. Thus, the total coarse time steps over $[0,T]$ are $\Nctime = \sum_{p=1}^{\NParSub} \Nctimep{p}$, and the total fine time steps are $\Nftime = \sum_{p=1}^{\NParSub} \Nftimep{p}$. On the coarse scale, we partition  $[T_{p-1}, T_p]$ as $T_{p-1} = \ctimepj{p}{0} < \cdots < \ctimepj{p}{\Nctimep{p-1}} < \ctimepj{p}{\Nctimep{p}} = T_p$, as shown in Figure~\ref{fig:subdomains_and_discretizations}, with $\Intcpj{p}{\hat n} = [\ctimepj{p}{\hat n-1}, \ctimepj{p}{\hat n}]$. Let $\Taucp{p} = \lbrace \Intcpj{p}{1}, \cdots, \Intcpj{p}{\hat n}, \cdots, \Intcpj{p}{\Nctimep{p}} \rbrace$. On the fine scale we introduce a finer discretization $T_{p-1} = \ftimepj{p}{0} < \cdots < \ftimepj{p}{\Nftimep{p-1}} < \ftimepj{p}{\Nftimep{p}} = T_p$ as shown in Figure~\ref{fig:subdomains_and_discretizations}. We let  $\Intfpj{p}{n} = [ \ftimepj{p}{n-1}, \ftimepj{p}{n} ]$ and  $\Taufp{p} = \lbrace \Intfpj{p}{1}, \cdots, \Intfpj{p}{n}, \cdots, \Intfpj{p}{\Nftimep{p}} \rbrace$.

\begin{figure}
\centering
\includegraphics[width=.77\textwidth]{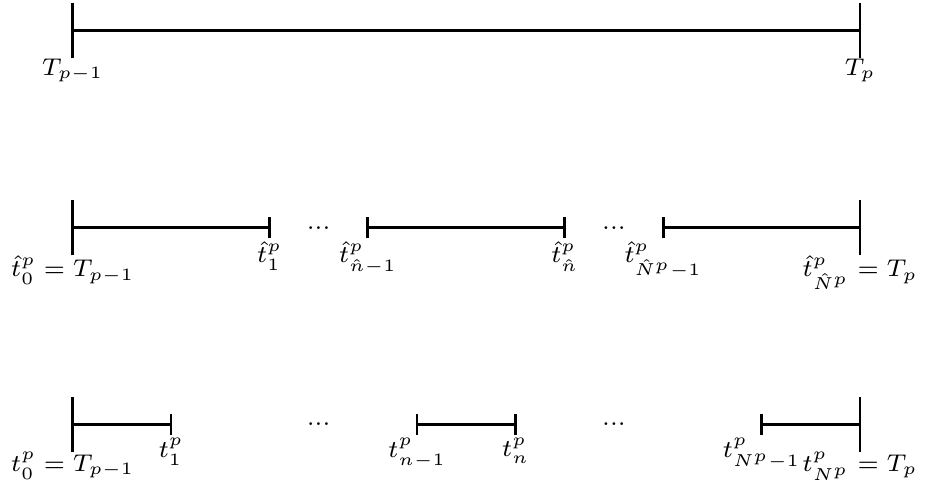}
\caption{Top: Subdomain $[T_{p-1},T_p]$. Middle: Stage 1 discretization for $[T_{p-1},T_p]$.  Bottom: Stage 2 discretization for $[T_{p-1},T_p]$.}\label{fig:subdomains_and_discretizations}
\end{figure}

We now use the implicit Euler method on the coarse time step $\Intcpj{p}{\hat n}$: given the approximate solution $\csolnj{n-1}$ at $\ctimepj{p}{\hat n-1}$, compute approximation $\csolnj{n} \in \Vspace{\Qcspace}$ at $\ctimepj{p}{\hat n}$by,
\begin{equation}
  \label{eq:be_step_coarse}
  (\csolnj{n},v ) = (\csolnj{n-1},v) + \cdTpj{p}{n}\left[-a(\csolnj{n},v) + l(v) \right],
\end{equation}
for all $v \in  \Vspace{\Qcspace}$,where  $\cdTpj{p}{n} = (\ctimepj{p}{n}-\ctimepj{p}{n-1})$. Moreover, in the context of Parareal algorithm, we have $\csolnj{n-1} \in \Vspace{\Qcspace}$ for all ${\hat n}$ except possibly for ${\hat n} = 1$, where $\csolnj{0}$ (the initial condition for the coarse scale solver) may belong to $\Vspace{q}$.

Similarly, the implicit Euler method on the fine  time step $\Intfpj{p}{n}$ is: given the approximate solution $\fsolnj{n-1}$ at $\ftimepj{p}{n-1}$, compute approximation $\fsolnj{n} \in \Vspace{q}$ at $\ftimepj{p}{n}$ by,
\begin{equation}
  \label{eq:be_step_fine}
  (\fsolnj{n},v ) = (\fsolnj{n-1},v) + \fdTpj{p}{n}\left[-a(\fsolnj{n},v) + l(v) \right],
\end{equation}
for all $v \in  \Vspace{q}$, where  $\fdTpj{p}{n} = (\ftimepj{p}{n}-\ftimepj{p}{n-1})$. In the context of Parareal algorithm, we have $\fsolnj{n-1} \in \Vspace{q}$ for all $n$ except for $n=1$ for which  $\fsolnj{0}$ (the initial condition for the fine scale solver) belongs to $\Vspace{\Qcspace}$.

In the TPA, we assume that equations \ref{eq:be_step_coarse} and \ref{eq:be_step_fine} are solved exactly (up to numerical precision).

\subsection{Space-Time Parallel Algorithm (STPA)}
\label{sec:STPA}

The STPA introduces spatial parallelization (in addition to temporal parallelization) by using a Schwarz domain decomposition iteration at the fine scale. It is assumed in the spirit of the Parareal algorithm  that the coarse scale is solved implicitly without needing any spatial domain decomposition.

\subsubsection{Schwarz Domain Decomposition at a single time step}

  The additive Schwarz Domain decomposition iteration is employed to solve the  spatial problem in \eqref{eq:be_step_fine}. 
  Similar to \S \ref{sec:time_space_discretization}, we omit the subscript $p$ where this is clear.
  We rewrite the problem  \eqref{eq:be_step_fine}  on subinterval $\Intfpj{p}{n}$ as: find $\fsolnj{n}\in \Vspace{\Qfspace}$ such that,

\begin{equation}
  \label{eq:be_step_bilin form}
  B^n(\fsolnj{n},v) = \lspace{n}(v)
\end{equation}
for all $v \in \Vspace{q}$, where the bilinear form $B^n(\cdot, \cdot)$ and the linear form $\lspace{n}(\cdot)$ are defined by,
\begin{equation}
\label{eq:bilin_lin_forms_dd}
\begin{aligned}
B^n(u,v) &= (u,v) + \fdTpj{p}{n} a(u,v),\\
\lspace{n}(v) &= (\fsolnj{n-1},v) + \fdTpj{p}{n} l(v).
\end{aligned}
\end{equation}

Assume that we have $\NDDSub$ overlapping subdomains $\Omega_1, \cdots, \Omega_{\NDDSub}$ on  $\Omega$, such that for any  $\Omega_i$, there exists a $\Omega_j$, $i \neq j$, for which $\Omega_i \cap \Omega_j \neq \emptyset$ and $\cup_i \Omega_i = \Omega$.
Let $\mathcal{T}_{h,i} \equiv \mathcal{T}_{h} \cap \Omega_i$. We further assume that the triangulation $\mathcal{T}_{h}$ is consistent with the domain decomposition  in the sense that for any $T_s \in \mathcal{T}_{h}$,  if $T_s \cap \Omega_j \neq \emptyset$ then $T_s \subset \Omega_j$.
We let $(\cdot, \cdot)_{ij}$ represent the $L_2(\Omega_i \cap \Omega_j)$ inner product.
We denote by  $B^n_i(\cdot, \cdot)$ the  restriction  of $B^n(\cdot, \cdot)$  to  $\Omega_i$ and $B^n_{ij} (\cdot, \cdot)$  the restriction of $B^n(\cdot, \cdot)$ to $\Omega_i \cap \Omega_j$. Similarly, we let $\lspace{n}_i(\cdot)$ be the restriction of $\lspace{n}(\cdot)$ to $\Omega_i$.
The additive Schwarz method for the solution of \eqref{eq:be_step_bilin form}  is given in Algorithm \ref{alg:additive_basic}. 
The domain decomposition solution for $n =1, \ldots, \Nftimep{p}$ are  $\fsolnjk{n}{\NDDIte} =$ ASDD($B^n, \lspace{n}, \NDDIte, \NDDSub,\fsolnjk{n-1}{\NDDIte} $). Note that the initial guess to the algorithm is the solution from the previous time step $\fsolnjk{n-1}{\NDDIte}$.
In the algorithm $\Vspacedd{\Qfspace}{i}$ represents the space of continuous piecewise polynomials of degree $\Qfspace$ on $\mathcal{T}_{h,i}$ such that if $U \in \Vspacedd{\Qfspace}{i}$ then $U = 0$ on $\partial \Omega_i$, and  $\Vspaceddbc{\Qfspace}{i}{k_s}$ represents the space of continuous piecewise polynomials of degree $\Qfspace$ on $\mathcal{T}_{h,i}$ such that if $U \in \Vspaceddbc{\Qfspace}{i}{k_s}$ then $U = 0$ on $\partial \Omega$ and $U = \fsolnjk{n}{k_s}$ on $\partial \Omega_i \setminus \partial \Omega$.
The  Richardson parameter $\tau$, is needed to ensure that the iteration converges \cite{Tarek2008}.

\begin{algorithm}[H]\caption{Overlapping additive Schwarz domain decomposition}
\label{alg:additive_basic}
\begin{algorithmic}
\Procedure{ASDD}{$B^n, \lspace{n}, \NDDIte, \NDDSub,\fsolnjk{n}{0}$}
\For{$k=0, 1, 2, \dots, \NDDIte-1$ }
\For{$i= 1, 2, \dots, \NDDSub$ }
\State Find $ \fsolnikdd{i}{k_s+1} \in \Vspaceddbc{\Qfspace}{i}{k_s}$ such that \\ \qquad\qquad
  \begin{equation} \label{eq:additive_basic_local}
    B^n_i \big( \fsolnikdd{i}{k_s+1},v \big)= \lspace{n}_i(v), \quad \forall v \in \Vspacedd{\Qfspace}{i}.
  \end{equation}
\State Let
  \begin{equation} \label{eq:additive_basic_global}
     \fsolnjk{n}{k_s+1} = (1 - \tau \NDDSub)\fsolnjk{n}{k_s}
                 + \tau \left ( \sum_{i=1}^{\NDDSub } {\Pi_i} \fsolnikdd{i}{k_s+1} \right )
     \;\hbox{where}\quad
     \Pi_i  \fsolnikdd{i}{k_s+1} = \left \{
    \begin{gathered} \begin{aligned}
      &\fsolnikdd{i}{k_s+1}, \; &&\hbox{on } \overline{\Omega}_i, \\
      &\fsolnjk{n}{k_s},                   &&\hbox{on } \Omega \backslash \overline{\Omega}_i.
    \end{aligned} \end{gathered}  \right .
  \end{equation}
%
\EndFor
\EndFor

\State \Return $\fsolnjk{n}{\NDDIte}$
\EndProcedure
\end{algorithmic}
\end{algorithm}

\subsubsection{The global STPA solution}

In the context of the Parareal algorithm (Algorithm \ref{alg:parareal_variational}), we denote the solution by $\fsolnpkl{p}{k_t}{\NDDIte}$, where we have now added $\NDDIte$ to the superscript to indicate the dependence of the solution on domain decomposition iterations at the $k_t$th Parareal iteration. The global STPA solution on $[0,T]$ after  $k_t$ Parareal iterations and $\NDDIte$ domain decomposition iterations is defined as $\fsolnkl{k_t}{\NDDIte} = \fsolnpkl{p}{k_t}{\NDDIte}$ for $t \in (T_{p-1}, T_p]$.

\section{Variational methods and preliminary \emph{a posteriori} error analysis}
\label{sec:var_mthds}

In this section, we introduce  basic elements needed for the error analysis of the numerical value of the QoI (see \eqref{eq:qoi_def}) computed using the fine scale solution from either the the Time Parallel Algorithm or the Space-Time Parallel Algorithm.
The main tools used in the analysis are adjoint problems coupled with an appropriate definition of residuals. This kind of analysis relies on the solution being in a variational format. Hence, we present the implicit Euler method as a variational method in \S \ref{sec:be_as_dg0}, and point out a fundamental property of variational methods which is used subsequently in the analysis of TPA and STPA. 


\subsection{Variational methods}
We define two variational methods; the continuous Galerkin method and the discontinuous Galerkin method for the fine scale; the extension to the coarse scale is analogous. Note that the usage of ``continuous'' or ``discontinuous'' refers to the property of the solution in time only; in space we always use continuous finite elements.

On the space-time slab $S^p_{n} = \Omega \times \Intfpj{p}{n}$ we define the space
\begin{equation}
   \label{Vspacetimef}
   \Vspacetime{n}{\Qftime}{\Qfspace} := \{ w(x,t):  w(x,t) = \sum_{j=0}^{\Qctime} t^j v_j(x), v_j \in \Vspace{\Qfspace}, (x,t) \in S^p_{n}  \},
\end{equation}

The \textbf{continuous Galerkin method,  cG($\Qftime$)} is: find $\fsolnCG$ such that it is continuous on $S^p_{n}$, its restriction to $S^p_{n}$ belongs to $\Vspacetime{n}{\Qftime}{\Qfspace}$  and also satisfies,
\begin{equation}
\label{eq:fine_solv_one_dt_cg}
\finv{ (\dfsolnCG,v)}{p}{n} = \finv{-a(\fsolnCG,v)+ l(v)}{p}{n}
\end{equation}
for all $v \in \Vspacetime{n}{\Qftime-1}{\Qfspace}$.

We denote the jump across $\ftimepj{p}{n}$ as $[w]_{n,p} = w_{n,p}^+ - w_{n,p}^-$, where $w_{n,p}^{\pm} = \lim_{s \rightarrow (\ftimepj{p}{n})^{\pm}}w(s)$.
The \textbf{discontinuous Galerkin method, dG($\Qftime$)} is: find $\fsolnCG$  its restriction to $S^p_{n}$ belongs to $\Vspacetime{n}{\Qftime}{\Qfspace}$  and also satisfies,
\begin{equation}
\label{eq:fine_solv_one_dt_dg}
\finv{ (\dfsolnDG,v)}{p}{n} + ([\fsolnDG]_{n-1,p},v_{n-1,p}^+)  = \finv{-a(\fsolnDG,v)+ l(v)}{p}{n}
\end{equation}
for all $v \in \Vspacetime{n}{\Qftime}{\Qfspace}$. 
Note that the solution may be discontinuous in at the time nodes $\ftimepj{p}{n}$.

\subsection{ A posteriori error analysis for variational methods on a single time subdomain}
Variational methods satisfy an important property that underpins the \emph{a posteriori} error analysis presented here. Let $\adjoint \in L^2(T_{p-1}, T_p; H^1_0(\Omega))$ satisfy,
\begin{equation}
\label{eq:generic_phi_p}
  (-\dot{\adjoint}, v) = -a(v, \adjoint) \quad \forall v \in H^1_0(\Omega) \text{ and } t \in (T_{p-1},T_p]
\end{equation}
Equation \eqref{eq:generic_phi_p} is referred to as an adjoint equation. 
Next we quantify the accumulation of error contributions.
\begin{lem}

 Let $\ferrVAR = \soln - \fsolnVAR$, where  $\soln$ is the solution to \eqref{eq:pde}, $\mathrm{Var} \in \{ CG, DG \}$ and  $\fsolnVAR$ is either the  cG($\Qftime$) or the dG($\Qftime$) solution on a time subdomain $p$. 
Then,
\begin{equation}
\label{eq:var_prop}
(\adjoint, \ferrVAR)(T_p) = (\adjoint, \ferrVAR)(T_{p-1}) + \VARrespuphi{p}{\fsolnVAR}{\adjoint}
\end{equation}
where $\adjoint$ is defined in \eqref{eq:generic_phi_p} and $\VARrespuphi{p}{\fsolnVAR}{\adjoint}$ is the adjoint-weighted space-time residual of the  solution $\fsolnVAR$ on time subdomain $p$ defined as
\begin{equation}
\label{eq:res_cg_dg}
  \VARrespuphi{p}{\fsolnVAR}{\adjoint} = \begin{cases}
  \sum_{n=1}^{{N}_p} \left[ \finv{l(v) -  a(\fsolnCG, v}{p}{n}
          -\finv{( \dfsolnCG, v)}{p}{n} \right], \qquad &\text{ if } \mathrm{Var} = CG,\\
  \sum_{n=1}^{{N}_p} \left[ \finv{l(v) -  a(\fsolnDG, v}{p}{n}
          -\finv{( \dfsolnDG, v)}{p}{n} - ([\fsolnDG]_{n-1,p},\adjoint_{n-1,p}^+) \right], \; &\text{ if } \mathrm{Var} = DG.
  \end{cases}
\end{equation}

\end{lem}

\begin{proof}
We prove property \eqref{eq:var_prop} for the dG method, the proof for the cG method is similar. In the notation of the dG method, we have 
\begin{equation}
\label{eq:not_para_dg}
\ferrVAR(T_{p-1}) = \ferrVAR^+(\ftimepj{p}{0}) \qquad \text{and} \qquad \ferrVAR(T_{p}) = \ferrVAR^-(\ftimepj{p}{{N}_p}).  
\end{equation}
Consider
\begin{equation}
\label{eq:lem_err_acc_1}
\finv{ (-\dot{\adjoint},\ferrDG)}{p}{n}  = -\finv{a(\ferrDG,\adjoint)}{p}{n}.
\end{equation}
The left-hand side is
\begin{equation}
\label{eq:lem_err_acc_2}
\finv{ (-\dot{\adjoint},\ferrDG)}{p}{n}   = - (\adjoint, \ferrDG^-)(\ftimepj{p}{n})  + (\adjoint, \ferrDG^+)(\ftimepj{p}{n-1}) + \finv{(\adjoint, \dsoln - \dfsolnDG)}{p}{n}.
\end{equation}
Combining \eqref{eq:lem_err_acc_1} and \eqref{eq:lem_err_acc_2} with \eqref{eq:pde}  leads to
\begin{equation}
\label{eq:err_sing_int}
\begin{aligned}
(\fadj, \ferrDG^-)(\ftimepj{p}{n})  &= (\adjoint, \ferrDG^+)(\ftimepj{p}{n-1})
+\finv{l(v) -  a(\fsolnDGp{p}, v}{p}{n}
          -\finv{( \dfsolnDG, v)}{p}{n},\\
          &= (\adjoint, \ferrDG^-)(\ftimepj{p}{n-1})
          -(\adjoint_{n-1,p}^+, [\fsolnDG]_{n,p})
+\finv{l(v) -  a(\fsolnDGp{p}, v}{p}{n}
          -\finv{( \dfsolnDG, v)}{p}{n}
\end{aligned}
\end{equation}
Summing \eqref{eq:err_sing_int} over $n = 1$ to $n = \Nftimep{p}$ and using \eqref{eq:not_para_dg},
\begin{equation}
\begin{aligned}
(\adjoint, \ferrDG)(T_p) &= (\adjoint, \ferrDG)(T_{p-1}) + \frespuphi{p}{\fsolnDG}{\adjoint},
\end{aligned}
\end{equation}
which is the desired result. 
\end{proof}
Note that a similar result, with similar derivation, also holds for the coarse scale.

\subsection{\emph{A posteriori} error analysis for implicit Euler on a single time subdomain}
\label{sec:be_as_dg0}

It is well known that the implicit Euler method with a certain choice of quadrature is nodally equivalent to the $dG(0)$ method and hence may be considered a variational method~\cite{eehj_book_96}. For completeness, we show this equivalence for the $dG(0)$ in \eqref{eq:fine_solv_one_dt_dg} with the implicit Euler method in \eqref{eq:be_step_fine}. For the $dG(0)$ method, we have $\dfsolnDG = 0$, since the solution is constant in time over $\Intfpj{p}{n}$ for a fixed spatial point. Moreover, $\Vspacetime{n}{0}{\Qfspace}$ is simply $\Vspace{\Qfspace}$. 
 Now, using these facts, along with the right-hand rule of evaluating quadrature in \eqref{eq:fine_solv_one_dt_dg},
\begin{equation}
  ( (\fsolnDG)_{n,p}^- - (\fsolnDG)_{n,p}^+,v_{n-1,p}^+)  =  \fdTpj{p}{n} \left[ a((\fsolnDG)_{n,p}^-,v_{n,p}^-) + l(v) \right]
\end{equation}
Now, since $v \in \Vspacetime{n}{0}{\Qfspace}$ is constant in time, we have  $v_{n-1,p}^+ = v_{n,p}^- = v$. Finally, identifying $(\fsolnDG)_{n,p}^-$ with $\fsolnj{n}$ and $(\fsolnDG)_{n,p}^+$ with $\fsolnj{n-1}$ completes the equivalence of the two equations.


Now we define  coarse and fine scale analogues to  \eqref{eq:res_cg_dg} for the implicit Euler method.
The fine scale residuals are,
\begin{align}
\frespnarg{p}{n}{w}{v} &=  \finv{l(v) -  a(w, v}{p}{n}
          - ([w]_{n-1,p},v^+), \label{eq:fine_scale_residual_I} \\
\frespuphi{p}{w}{v}&= \sum_{n=1}^{{N}_p}\frespnarg{p}{n}{w}{v}. \label{eq:fine_scale_residual}
\end{align}
The coarse scale residuals are,
\begin{align}
\crespnarg{p}{\hat{n}}{w}{v} &=  \cinv{l(v) -  a(w, v}{p}{\hat{n}}
          - ([w]_{\hat{n}-1,p},v^+), \label{eq:coarse_scale_residual_I} \\
\crespuphi{p}{w}{v}&= \sum_{\hat{n}=1}^{\hat{N}_p}\crespnarg{p}{\hat{n}}{w}{v}. \label{eq:coarse_scale_residual}
\end{align}
Notice that there is no time derivative in the residuals because, as discussed earlier, those terms are 0 for the $dG(0)$ method. 

The analogues of property \eqref{eq:var_prop} for the coarse and fine scale solutions in Algorithm~\ref{alg:parareal_variational} follow immediately.
Let $\fsolnp{p} = \fsolvernoargs[\fsolnp{p}(T_{p-1})]$ be the fine-scale solution on time domain $p$, and $\ferrp{p} = \soln - \fsolnp{p}$. Then,
\begin{equation}
\label{eq:fres}
(\adjoint, \ferrp{p})(T_p) = (\adjoint, \ferrp{p})(T_{p-1}) + \frespuphi{p}{\fsolnp{p}}{\adjoint}.
\end{equation}
Let  $\csolnp{p} = \csolvernoargs[\csolnp{p}(T_{p-1})]$  be the coarse scale solution on time domain $p$. Let $\cerrp{p} = \soln - \csolnp{p}$.
Then,
\begin{equation}
\label{eq:cres}
(\adjoint, \cerrp{p})(T_p) = (\adjoint, \cerrp{p})(T_{p-1}) + \crespuphi{p}{\csolnp{p}}{\adjoint}.
\end{equation}

\section{\emph{A posteriori} error analysis for the Time Parallel  Algorithm (TPA)}
\label{sec:temporal_errors}

In this section we derive  estimates for the error in the QoI computed from the fine scale solution of the TPA presented in \S \ref{sec:TPA}. That is, we seek to estimate $\mathcal{Q}(\soln - \fsolnk{\NParIte}) = ( \psi, \soln(T) -\fsolnk{\NParIte}(T))$. 
The analysis for TPA is quite similar to the analysis for ODEs appearing in \cite{chaudhry2016posteriori}, except that the analysis for PDEs includes the error in the initial conditions, which is assumed to be zero for ODEs.  

\subsection{Coarse and fine scale adjoint problems}

The coarse scale adjoint problem is: find $\cadj \in L^2(0,T; H^1_0(\Omega))$ such that
\begin{equation}
\label{eq:coarse_adjoint}
\left .
\begin{gathered}\begin{aligned}
(-\dot{\cadj}, v) &= -a(v, \cadj) \quad \forall v \in H^1_0(\Omega) \text{ and } t \in (0,T], \\
       \cadj(T) &=\psi.
\end{aligned}\end{gathered}
\qquad \right \}
\end{equation}
Here, the term ``coarse'' indicates that in numerical experiments, this adjoint is approximated using a coarse (spatial and temporal) scale space.

For $p = 1, 2, \ldots, \NParSub$, the fine scale  adjoint problem for the $p$th time domain $[T_{p-1}, T_p]$ is: find $\fadjp{p} \in  L^2(T_{p-1}, T_p; H^1_0(\Omega))$ such that
\begin{equation}
\label{eq:fine_adjoint_equation}
\left .
\begin{gathered}\begin{aligned}
(-\dfadjp{p},v) &= -a(v,\fadjp{p})   \quad  \forall v \in H^1_0(\Omega) \text{ and } t \in (T_{p-1}, T_p], \\
 \fadjp{p}(T_p) &= \cadj (T_p).
\end{aligned}\end{gathered}
\qquad \right \}
\end{equation}
Notice that the fine scale adjoint problem on the $p$th time subdomain receives initial conditions from the coarse scale adjoint problem and  therefore the adjoint on the fine temporal scale can be solved independently on each of the $p$ time domains. However, while reducing solution cost, the resulting discontinuities in the fine temporal scale adjoint solution occurring at the boundaries of temporal subdomains give rise to  an additional contribution to the error. In numerical experiments experiments, the fine scale adjoints are approximated using a fine (relative to the space used for the approximation of the coarse scale adjoints) scale space.

We seek the error in the fine scale solution below. However, since coarse scale solutions are used as initial conditions for fine scale integration on each subdomain, the analysis indicates the expected interaction between errors at the two different scales.


\subsection{Error representations}



Let $\ferrk{k_t} = \soln - \fsolnk{k_t}$ and $\cerrk{k_t} = \soln - \csolnk{k_t}$.
\begin{lem}
\begin{equation}
\label{eq:par_err_rep_without_aux_errs}
\begin{aligned}
(\psi, \ferrk{k_t}(T)) &=
\sum_{p=1}^{\NParSub}\frespuphi{p}{\fsolnpk{p}{k_t}}{\fadjp{p}} +   (\fadjp{1}, \cerrpk{1}{k_t})(0) \\
& + \sum_{p=2}^{\NParSub}\bigg[ (\fadjp{p} - \cadj,\cerrpk{p-1}{k_t}) + (\fadjp{p} - \cadj, \csolnpk{p-1}{k_t} - \csolnpk{p}{k_t}) + (\cadj, \fsolnpk{p-1}{k_t} - \fsolnpk{p}{k_t})\bigg] (T_{p-1}).
\end{aligned}
\end{equation}
\end{lem}

\begin{proof}
By definition of $\fsolnk{k_t}$, we have $\fsolnk{k_t}(T) = \fsolnpk{\NParSub}{k_t}(T) = \fsolnpk{\NParSub}{k_t}(T_{\NParSub})$. From  \eqref{eq:fres} and
$\fadjp{p}(T_p) = \cadj(T_p)$ (see \eqref{eq:fine_adjoint_equation}) we have
\begin{equation}
(\cadj, \ferrpk{p}{k_t})(T_p) = (\fadjp{p}, \ferrpk{p}{k_t})(T_{p-1}) + \frespuphi{p}{\fsolnpk{p}{k_t}}{\fadjp{p}}.
\end{equation}
Summing over all time subdomains from $p=1$ to $\NParSub$, and isolating $(\psi, \ferrpk{\NParSub}{k_t}(T)) = (\psi, \ferrk{k_t}(T))$  leads to

\begin{equation}
\begin{aligned}
(\psi, \ferrk{k_t}(T)) &= -\sum_{p=1}^{\NParSub-1} (\cadj, \ferrpk{p}{k_t})(T_p) + \sum_{p=2}^{\NParSub} (\fadjp{p}, \ferrpk{p}{k_t})(T_{p-1})   + \sum_{p=1}^{\NParSub}\frespuphi{p}{\fsolnpk{p}{k_t}}{\fadjp{p}} + (\fadjp{1}, \ferrpk{1}{k_t})(0).
\end{aligned}
\end{equation}
Using $\fsolnpk{p}{k_t}(T_{p-1}) = \csolnpk{p}{k_t}(T_{p-1})$ (see Algorithm \ref{alg:parareal_variational}),
\begin{equation}
  (\psi, \ferrk{k_t}(T)) = -\sum_{p=1}^{\NParSub-1} (\cadj, \ferrpk{p}{k_t})(T_p) + \sum_{p=2}^{\NParSub} (\fadjp{p}, \cerrpk{p}{k_t})(T_{p-1})   + \sum_{p=1}^{\NParSub}\frespuphi{p}{\fsolnpk{p}{k_t}}{\fadjp{p}}+ (\fadjp{1}, \cerrpk{1}{k_t})(0).
\end{equation}
Rearranging, adding and subtracting terms as necessary, we arrive at

\begin{equation}
\begin{aligned}
(\psi, \ferrk{k_t}(T)) &= \sum_{p=1}^{\NParSub}\frespuphi{p}{\fsolnpk{p}{k_t}}{\fadjp{p}} +   (\fadjp{1}, \cerrp{1})(0) \\
&+  \sum_{p=2}^{\NParSub}\bigg[ (\fadjp{p} - \cadj,\cerrpk{p-1}{k_t}) + (\fadjp{p} - \cadj, \csolnpk{p-1}{k_t} - \csolnpk{p}{k_t}) + (\cadj, \fsolnpk{p-1}{k_t} - \fsolnpk{p}{k_t})\bigg](T_{p-1}).
\end{aligned}
\end{equation}

\end{proof}


The terms $(\fadjp{p} - \cadj,\cerrpk{p-1}{k_t})(T_{p-1})$ in the error representation \eqref{eq:par_err_rep_without_aux_errs} are not directly computable as they contain the unknown error $\cerrp{p-1}$. We solve auxiliary adjoint problems to account for these error terms. As noted earlier, this is the tradeoff between solving the fine-scale adjoint problem from $T$ to $0$ and solving $P$ adjoint problems on each subdomain independently using the coarse scale adjoint problem as ``initial'' conditions.  Consider $(\NParSub-1)$ auxiliary (coarse scale) QoIs as
\begin{align}
  \Psi(u) = (\fadjp{p} - \cadj, u),
\end{align}
and $(\NParSub-1)$ auxiliary adjoint problems: find $\dauxadjp{p} \in L^2(0,T_{p-1}; H^1_0(\Omega))$ such that
\begin{equation}
\label{eq:auxiliary_adjoint_equation}
\left .
\begin{gathered}\begin{aligned}
     -\dauxadjp{p} &= -a(v, \auxadjp{p} ), \qquad \quad  \forall v \in H^1_0(\Omega) \text{ and } t \in (0,T_{p-1}]\\
\auxadjp{p}(x,T_{p-1}) &= \fadjp{p}(x,T_{p-1}) - \cadj(x,T_{p-1}).
\end{aligned}\end{gathered}
\qquad \right \}
\end{equation}
Replacing $\adjoint$ by $\auxadjp{p}$ in \eqref{eq:cres} on $k$th time subdomain we have

\begin{equation}
  (\auxadjp{p}, \cerrpk{k}{k_t})(T_{k}) = (\auxadjp{p}, \cerrpk{k}{k_t})(T_{k-1}) + \crespuphi{k}{\csolnpk{k}{k_t}}{\auxadjp{p}}.
\end{equation}
Summing from $k=1$ to $k = p-1$ and using \eqref{eq:auxiliary_adjoint_equation},
\begin{equation}
\label{eq:aux_err}
(\fadjp{p} - \cadj,\cerrpk{p-1}{k_t})(T_{p-1})
= \sum_{k=2}^{p-1}(\auxadjp{p}, \csolnpk{k-1}{k_t} - \csolnpk{k}{k_t})(T_{k-1})
  + \sum_{k=1}^{p-1} \crespuphi{k}{\csolnpk{k}{k_t}}{\auxadjp{p}} 
   + (\auxadjp{p}, \cerrpk{1}{k_t})(0).
\end{equation}


Combining \eqref{eq:par_err_rep_without_aux_errs} and \eqref{eq:aux_err} yields

\begin{equation}
  \begin{aligned}
  (\psi, \ferrk{k_t}(T)) &= \sum_{p=1}^{\NParSub}\frespuphi{p}{\fsolnpk{p}{k_t}}{\fadjp{p}} + (\fadjp{1}, \cerrpk{1}{k_t})(0) \\
  &+  \sum_{p=2}^{\NParSub}\bigg[
                     \sum_{k=2}^{p-1}(\auxadjp{p}, \csolnpk{k-1}{k_t} - \csolnpk{k}{k_t})(T_{k-1})
                   + \sum_{k=1}^{p-1} \crespuphi{k}{\csolnpk{k}{k_t}}{\auxadjp{p}} + (\auxadjp{p}, \cerrpk{1}{k_t})(0) \bigg] \\
  &+ \sum_{p=2}^{\NParSub}\bigg[(\fadjp{p} - \cadj, \csolnpk{p-1}{k_t} - \fsolnpk{p}{k_t})
                              + (\cadj, \fsolnpk{p-1}{k_t} - \fsolnpk{p}{k_t})\bigg](T_{p-1}).
  \end{aligned}
\end{equation}
We summarize as Theorem \ref{thm:par_err} below.

\begin{thm}
\label{thm:par_err}
\begin{equation}
\label{eq:err_decomp_parareal}
  (\psi, \ferrk{k_t}(T)) = \fterror + \auxerror + \cterror  + \Kterror,
\end{equation}
where
\begin{eqnarray}
\label{eq:fine_scale_temporal_error_comps}
\fterror  &=& \sum_{p=1}^{\NParSub} \frespuphi{p}{\fsolnpk{p}{k_t}}{\fadjp{p}} + (\fadjp{1},\cerrpk{1}{k_t}))(0), \nonumber \\
\auxerror &=& \sum_{p=2}^{\NParSub} \left[
                     \sum_{k=2}^{p-1}(\auxadjp{p}, \csolnpk{k-1}{k_t} - \csolnpk{k}{k_t})(T_{k-1})
                   + \sum_{k=1}^{p-1} \crespuphi{k}{\csolnpk{k}{k_t}}{\auxadjp{p}} + (\auxadjp{p}, \cerrpk{1}{k_t})(0) \right], \nonumber \\
\cterror &=& \sum_{p=2}^{\NParSub} (\fadjp{p} - \cadj, \csolnpk{p-1}{k_t} - \csolnpk{p}{k_t}) (T_{p-1}), \nonumber \\
\Kterror &=& \sum_{p=2}^{\NParSub} (\cadj, \fsolnpk{p-1}{k_t} - \fsolnpk{p}{k_t} )(T_{p-1}). \nonumber
\end{eqnarray}
\end{thm}

We identify the following components that comprise the total error.
\begin{enumerate}
  \item $\fterror$ is the ``standard'' fine scale discretization error contribution.
  \item $\auxerror$ is the temporal auxiliary error contribution arising from the discontinuities in the fine scale adjoint solution at the synchronization points,
  \item $\cterror$ is the temporal coarse scale error contribution from the discontinuities in the coarse scale solution and the fine scale adjoint solution at synchronization points,
  \item $\Kterror$ is the temporal iteration error contribution arising from the difference between the (fine scale) solutions at the synchronization points.
\end{enumerate}
Numerical examples demonstrating the accuracy of this error estimator are provided in section \ref{sec:numerical_results_Par}. While the analysis above concerns the fine scale error, we note that the coarse scale error may also be estimated as described in  \ref{sec:appendix_coarse_error}.

\subsection{Extension to other time integrators}

The analysis in the previous section assumed the implicit Euler method as the time integration procedure for both the coarse and fine scale solutions. However, the analysis remains valid provided the time integration method used at the fine and coarse scales satisfies ~\eqref{eq:fres} and \eqref{eq:cres}. 
These requirements are satisfied for variational methods as demonstrated in \S \ref{sec:var_mthds}. Moreover, these requirements are also satisfied for 
a number of numerical schemes which may be interpreted as a variational methods, e.g., Crank Nicolson, BDF, IMEX, etc \cite{collins2014_LaxWendroff, collins2014_explicit, CEG+2015, Chaudhry2019, chaudhry2020posteriori}.


\section{\emph{A posteriori} analysis of the Space-Time Parallel Algorithm (STPA)}
\label{sec:spatial_errors}

In this section we derive error estimates for the error in the QoI computed from the fine scale solution of the STPA presented in \S \ref{sec:STPA}. That is, we seek to estimate $\mathcal{Q}(\soln - \fsolnpkl{p}{k_t}{\NDDIte}) = ( \psi, \soln(T) -\fsolnpkl{p}{k_t}{\NDDIte}(T))$.
Theorem \ref{thm:par_err} also gives the error in the STPA by replacing by replacing $\fsolnpk{p}{k_t}$ with $\fsolnpkl{p}{k_t}{\NDDIte}$. However, this result does not account for the error due to the spatial Schwarz domain decomposition iteration. In this section we extend the analysis to account for this source of error.

Let $\fsolnpjkl{p}{n}{k_t}{\NDDIte} = \fsolnpkl{p}{k_t}{\NDDIte}(\ftimepj{p}{n})$, where $\fsolnpjkl{p}{n}{k_t}{\NDDIte}$ is the domain decomposition solution to \eqref{eq:be_step_bilin form} on time domain $p$ for Parareal iteration $k_t$. More precisely, $\fsolnpjkl{p}{n}{k_t}{\NDDIte}$ is the domain decomposition solution to the problem: find $\fsolnpjk{p}{n}{k_t} \in  \Vspace{\Qfspace}$ such that
\begin{equation}
\label{eq:be_step_bilin form_pn}
  B^n(\fsolnpjk{p}{n}{k_t},v) = \lspace{n}(v),
\end{equation}
for all $v \in \Vspace{q}$, where the bilinear form $B^n(\cdot, \cdot)$ and the linear form $\lspace{n}(\cdot)$ are given in \eqref{eq:bilin_lin_forms_dd}, and $\fsolnpjk{p}{0}{k_t} = \fsolnpjkl{p}{0}{k_t}{\NDDIte} =  \csolnpk{p}{k_t}(x,T_{p-1})$.

For analysis purposes we also introduce an analytical solution: find $\solnpjk{p}{n}{k_t} \in H_0^1(\Omega)$ such that,
\begin{equation}
  \label{eq:be_step_bilin form_pn_analytical}
  B^n(\solnpjk{p}{n}{k_t},v) = \lspace{n}(v)  \qquad \forall v \in H^1_0(\Omega).
\end{equation}
Based on the discussion in section \ref{sec:be_as_dg0}, we may also consider $\solnpjk{p}{n}{k_t}$ in variational form.

\subsection{Overview of the strategy}

Trivially,
\begin{equation}
\label{eq:error_components_dd}
\begin{aligned}
\left( \soln(\ftimepj{p}{n}) - \fsolnpjkl{p}{n}{k_t}{\NDDIte}, \fadjp{p}(t_n) \right)
&= \underbrace{ \left( \soln(\ftimepj{p}{n})-\solnpjk{p}{n}{k_t}, \fadjp{p}(t_n)  \right)}_{\mathcal{E}_{I}}
 + \underbrace{\left( \solnpjk{p}{n}{k_t}- \fsolnpjkl{p}{n}{k_t}{\NDDIte}, \fadjp{p}(t_n) \right) }_{\mathcal{E}_{II}}
\end{aligned}
\end{equation}

\begin{enumerate}
  \item We estimate $\mathcal{E}_{I}$ by interpreting $\solnpjk{p}{n}{k_t}$ as a variational solution and using \emph{a posteriori} techniques developed in section \ref{sec:var_mthds}.
  \item We estimate $\mathcal{E}_{II}$ using the \emph{a posteriori} error analysis for overlapping domain decomposition presented in \cite{chaudhry2019posteriori}. This analysis allows us to further split $\mathcal{E}_{II}$ into iterative and discretization components.
\end{enumerate}

\subsection{Estimating $\mathcal{E}_{I}$}

In a derivation akin to that of \eqref{eq:err_sing_int}, interpreting $\solnpjk{p}{n}{k_t}$ as a space-time solution to the dG(0) method in time leads to,
\begin{equation}
\label{eq:err_be_inf_one_int}
\mathcal{E}_{I} = (\fadjp{p}(\ftimepj{p}{n}), \soln(\ftimepj{p}{n})  - \solnpjk{p}{n}{k_t})  = (\fadjp{p}(\ftimepj{p}{n}), \soln(\ftimepj{p}{n-1})  - \fsolnpjkl{p}{n-1}{k_t}{K_s})
+\frespnarg{p}{n}{\solnpjk{p}{n}{k_t}}{\fadjp{p}}
\end{equation}
where we also used the fact that $\solnpjk{p}{n}{k_t}(\ftimepj{p}{n-1}) = \fsolnpjkl{p}{n-1}{k_t}{K_s}$.

\subsection{Estimating $\mathcal{E}_{II}$}
\label{sec:be_rule_analysis}
The analysis of \cite{chaudhry2019posteriori} is used to estimate the error term $\mathcal{E}_{II}$. We define the global (spatial) adjoint problem: find $\adjdd{p}{n} \in H^1_0(\Omega)$ such that,
\begin{equation}
\label{eq:global_adj}
  B^n(v,\adjdd{p}{n}) = (\fadjp{p},v), \quad \forall v \in H^1_0(\Omega),
\end{equation}
and adjoint problems  on the spatial subdomains: find
 $\adjddadd{p}{n}{k_s}{i} \in H^{1}_0(\Omega_i)$ such that,
\begin{equation}\label{eq:additive_adjoints}
B_{i}^n\left(v, \adjddadd{p}{n}{k_s}{i} \right)
= \tau \sum_{j=1}^{P_s} \left\{ (\ddpsi{p}{n}{j},v)_{ij} - B_{ij}\left(v, \sum_{l=k_s+1}^{\NDDIte} \adjddadd{p}{n}{l}{i}  \right) \right \}, \quad \forall v \in H^{1}_0(\Omega_i),
\end{equation}
where $\ddpsi{p}{n}{i}$ is the restriction of  $\fadjp{p}(t_n)$ to $\Omega_i$,
and $B_i^n$ and $B_{ij}^n$ are the restrictions of $B^n$ to $\Omega_i$ and $\Omega_{ij}$ respectively. For a fixed $k_s$, the adjoint problems \eqref{eq:additive_adjoints} are independent for each $i$, so $\adjddadd{p}{n}{k_s}{i}$ may be computed backwards from $\NDDIte, \ \NDDIte-1, \ \NDDIte-2, \cdots, 1$ in parallel analogous to the solution strategy in the additive Schwarz Algorithm \ref{alg:additive_basic}.

\begin{thm}
\label{thm:dd_errs}
\begin{equation}
\label{eq:err_dd_only}
\mathcal{E}_{II} = \left( \solnpjk{p}{n}{k_t} - \fsolnpjkl{p}{n}{k_t}{\NDDIte}, \fadjp{p}(t_n) \right) = \fterrKspn{p}{n} + \fterrNspn{p}{n}
\end{equation}
where
\begin{align}
\fterrNspn{p}{n} &= \sum_{i=1}^{\NDDSub} \sum_{k_s = 1}^{\NDDIte}   \lspace{n}(\adjddadd{p}{n}{k_s}{i}) - B_i^n(\fsolpnkidd{p}{n}{i}{k_s},\adjddadd{p}{n}{k_s}{i})\\
\fterrKspn{p}{n} &= \lspace{n}(\adjdd{p}{n}) - B^n(\fsolnpjkl{p}{n}{k_t}{\NDDIte},\adjdd{p}{n})  - \fterrNspn{p}{n}
\end{align}
where $\fsolpnkidd{p}{n}{i}{k_s}$ are the spatial subdomain solutions on spatial subdomain $i$ defined by \eqref{eq:additive_basic_local}.
\end{thm}
The proof of Theorem~\ref{thm:dd_errs} is provided in \cite{chaudhry2019posteriori}.
The term $\fterrNspn{p}{n}$  quantifies the discretization error contribution (that is, due to using the spaces $\Vspaceddbc{\Qfspace}{i}{k_s}$) while the term $\fterrKspn{p}{n}$ quantifies the domain decomposition iteration error contribution (that is, due to using a finite number $\NDDIte$ iterations).

\subsection{Estimating $\mathcal{E}_{I} + \mathcal{E}_{II}$}

Combining \eqref{eq:err_be_inf_one_int} and \eqref{eq:err_dd_only},
\begin{equation}
\label{eq:one_int_dd_b}
\begin{aligned}
\mathcal{E}_{I} + \mathcal{E}_{II} &=
(\fadjp{p}(\ftimepj{p}{n}), \soln(\ftimepj{p}{n-1})  - \fsolnpjkl{p}{n-1}{k_t}{\NDDIte})
+\frespnarg{p}{n}{\fsolnpjkl{p}{n}{k_t}{\NDDIte}}{\fadjp{p}}\\
&= (\fadjp{p}(\ftimepj{p}{n}), \soln(\ftimepj{p}{n-1})  - \fsolnpjkl{p}{n-1}{k_t}{\NDDIte})
+\frespnarg{p}{n}{\solnpjk{p}{n}{k_t}}{\fadjp{p}}
 + \fterrKspn{p}{n} + \fterrNspn{p}{n},
\end{aligned}
\end{equation}
To provide a comparison with the two part construction above, following the direct approach prescribed in \eqref{eq:err_sing_int} we have,
\begin{equation}
\label{eq:one_int_dd}
\mathcal{E}_{I} + \mathcal{E}_{II} = (\fadjp{p}(\ftimepj{p}{n}), u(\ftimepj{p}{n})  - \fsolnpjkl{p}{n}{k_t}{\NDDIte})  = (\fadjp{p}(\ftimepj{p}{n}), u(\ftimepj{p}{n-1})  - \fsolnpjkl{p}{n-1}{k_t}{\NDDIte})
+\frespnarg{p}{n}{\fsolnpjkl{p}{n}{k_t}{\NDDIte}}{\fadjp{p}}.
\end{equation}
%
Comparing \eqref{eq:one_int_dd} and \eqref{eq:one_int_dd_b},
\begin{equation}
 \frespnarg{p}{n}{\fsolnpjkl{p}{n}{k_t}{\NDDIte}}{\fadjp{p}}
= \frespnarg{p}{n}{\solnpjk{p}{n}{k_t}}{\fadjp{p}}
 + \fterrKspn{p}{n} + \fterrNspn{p}{n}.
\end{equation}
Rearranging,
\begin{equation}
\label{eq:ParDD_computable_resid}
\frespnarg{p}{n}{\solnpjk{p}{n}{k_t}}{\fadjp{p}} = \frespnarg{p}{n}{\fsolnpjkl{p}{n}{k_t}{\NDDIte}}{\fadjp{p}}  -\fterrKspn{p}{n} - \fterrNspn{p}{n}.
\end{equation}
Since all terms on the RHS of \eqref{eq:ParDD_computable_resid} are computable, so is $\frespnarg{p}{n}{\solnpjk{p}{n}{k_t}}{\fadjp{p}}$.

\subsection{Summing contributions over time subdomain $p$}
Using \eqref{eq:fine_scale_residual}, the residual over the time domain $[T_{p-1},T_p]$ is
\begin{equation}
\label{eq:res_dd_one_time_sub_domain}
  \frespuphi{p}{\fsolnpjkl{p}{n}{k_t}{\NDDIte}}{\fadjp{p}} = \sum_{n=1}^{{N}_p} \frespnarg{p}{n}{\fsolnpjkl{p}{n}{k_t}{\NDDIte}}{\fadjp{p}} = \sum_{n=1}^{{N}_p} \left( \frespnarg{p}{n}{\solnpjk{p}{n}{k_t}}{\fadjp{p}}
 + \fterrKspn{p}{n} + \fterrNspn{p}{n} \right) = \fterrftp{p} +  \fterrNsp{p} +  \fterrKsp{p}
\end{equation}
where
\begin{equation}
  \fterrftp{p} = \sum_{n=1}^{{N}_p}\frespnarg{p}{n}{\solnpjk{p}{n}{k_t}}{\fadjp{p}}, \qquad \fterrKsp{p} = \sum_{n=1}^{{N}_p}\fterrKspn{p}{n}, \qquad \fterrNsp{p}= \sum_{n=1}^{{N}_p} \fterrNspn{p}{n}.
\end{equation}

\subsection{Error Representation for Space-Time Parallel Domain-Decomposition-Parareal Algorithm}

Let $\ferrkl{k_t }{\NDDIte} = \soln(T) - \fsolnpkl{p}{k_t}{\NDDIte}(T)$.

\begin{thm}
\label{thm:par_dd_err}
\begin{equation}
  (\psi, \ferrkl{k_t}{\NDDIte}(T)) = \fterrft +  \fterrNs +  \fterrKs + \auxerror + \cterror  + \Kterror,
\end{equation}
where $\auxerror, \cterror$ and  $\Kterror$ are defined in Theorem~\ref{thm:par_err} by \eqref{eq:fine_scale_temporal_error_comps} and,


\begin{equation}
  \fterrft  = \sum_{p=1}^{\NParSub} \fterrftp{p} + (\fadjp{1},\cerrpk{1}{k_t})(0), \qquad \fterrNs = \sum_{p=1}^{\NParSub} \fterrNsp{p}, \qquad \fterrKs = \sum_{p=1}^{\NParSub} \fterrKsp{p}.
\end{equation}
\end{thm}

\begin{proof}
The proof follows directly from Theorem~\ref{thm:par_err} by replacing $\fsolnpk{p}{k_t}$ with $\fsolnpkl{p}{k_t}{\NDDIte}$ and then using \eqref{eq:res_dd_one_time_sub_domain}.
\end{proof}

We identify the following additional error components.
\begin{enumerate}
  \item $\fterrft$ is the error contribution due to the use of implicit Euler time integration.
  \item $\fterrNs$ is the error contribution due to using a finite dimensional space for the solution of domain decomposition.
  \item $\fterrKs$ is the error contribution due to using a finite number, $\NDDIte$, iterations of domain decomposition.
\end{enumerate}

\section{Numerical results}
\label{sec:numerical_results}

We present numerical results to support the analyses of both section  \ref{sec:temporal_errors} and section \ref{sec:spatial_errors}, highlighting the accuracy of the error estimates developed. Accordingly, we first consider the error analysis for the Time Parallel  Algorithm  of section \ref{sec:numerical_results_Par} and then consider the effect of a further spatial domain decomposition iteration in the Space-Time Parallel  Algorithm of section \ref{sec:numerical_results_Par_DD}. The implicit Euler method is used in the fine and coarse scale solvers for these two sections. In section~\ref{sec:numerical_results_Par_CG} we demonstrate the accuracy of the error estimates when a different time integration, a cG method, is employed in the coarse and fine scale solvers.

 The accuracy of the \emph{a posteriori} error estimates is measured by the effectivity ratio $\gamma$, where
\begin{equation}
\label{eq:effectivity_ratio}
\gamma = \frac{\hbox{Estimated error}}{\hbox{True error}}.
\end{equation}
The bilinear and linear forms considered in the numerical examples are
\begin{equation}
\label{eq:temporal_errors_test_equation}
a(u,v) = -(\nabla u, \nabla v) \quad \hbox{and} \quad l(v) = (f(x),v)
\end{equation}
where
\begin{equation}
\label{eq:temporal_errors_test_rhs}
f(x) = \sin( \mu \pi x)(\mu^2 \pi^2 \cos(\nu \pi t) -  \nu \pi \sin(\nu \pi t)).
\end{equation}
This choice of $f(x)$ corresponds to the true solution $ u(x,t) = \cos(\nu \pi t) \sin(\mu \pi x)$.
In all examples, the final time $T=2$ and the quantity of interest at the final time $T$ is defined by the ``adjoint data''
\begin{equation}
\label{eq:temporal_errors_test_adjoint_data}
\psi(x) =  \left \{ \qquad
\begin{aligned}
10000 [ (x-0.2)^2  (x-0.6)^2 ], \qquad &0.2 <  x < 0.6,\\
0,  \qquad\qquad\qquad          & \text{ otherwise}.
\end{aligned}
\qquad \right.
\end{equation}
This choice of $\psi(x)$ corresponds to a local weighted average of the solution around $x=0.4$ at the final time.
The ratio of number of fine and coarse time steps is denoted by $\reffactime = \Nftime/\Nctime$. For the spatial discretization, the number of spatial finite elements at the two scales remained fixed, i.e., $\Nfspace = \Ncspace$, but different degrees of interpolation were employed according to whether the corresponding temporal integration was performed on the coarse or fine scales. A lower degree of spatial interpolation, $\Qcspace$, was employed when constructing the solution on the coarse  scale, and a higher degree of spatial interpolation, $\Qfspace$, was employed on the same spatial mesh when constructing the solution at the fine  scale, i.e., $\hat{q}_s < q_s$. The notation used in this section is also summarized in \ref{sec:appendix}.

Two different time steps were employed for the temporal integration, a coarse time step and a smaller, fine time step. The implicit Euler method was chosen as the time integrator for the coarse  and fine scale solvers $\csolvernoargs$ and $\fsolvernoargs$ in \S \ref{sec:numerical_results_Par} and \S \ref{sec:numerical_results_Par_DD}. A continuous Galerkin method was used to obtain the results in \S \ref{sec:numerical_results_Par_CG}.
%
The adjoint solutions in \eqref{eq:coarse_adjoint}, \eqref{eq:fine_adjoint_equation}, \eqref{eq:auxiliary_adjoint_equation}, \eqref{eq:global_adj} and \eqref{eq:additive_adjoints} need to be approximated. In all cases, the same temporal and spatial meshes were used as those used to compute the numerical solutions, however, higher degree approximations were employed to obtain accurate estimates. The adjoint solutions corresponding to \eqref{eq:coarse_adjoint}, \eqref{eq:fine_adjoint_equation}and \eqref{eq:auxiliary_adjoint_equation} were approximated on the space-time slab $S^p_{n}$  using the cG(3) method  in the space $\Vspacetime{n}{3}{3}$ (see \eqref{Vspacetimef} for the definition of this space). The adjoint solutions \eqref{eq:global_adj} and \eqref{eq:additive_adjoints}, needed for the analysis of STPA, were approximated by using a spatial finite element employing continuous piecewise cubic polynomial functions.

\subsection{Paralellism in time only: Time Parallel  Algorithm}
\label{sec:numerical_results_Par}

We first consider the \emph{a posteriori} error estimate given by Theorem \ref{thm:par_err}, Eq. \eqref{eq:fine_scale_temporal_error_comps} in section \ref{sec:temporal_errors}.  
%
The Parareal algorithm presents numerous discretization choices. We investigate the effect of the number of Parareal iterations in section \ref{sec:example_Par_iterations}, and the effect of the number of temporal subdomains in section \ref{sec:example_Par_subdomains}. We then consider the effect of the fine time scale and the coarse time scale in sections \ref{sec:example_Par_fine_time_scale} and \ref{sec:example_Par_coarse_time_scale} respectively. In sections \ref{sec:example_Par_iterations} to \ref{sec:example_Par_coarse_time_scale}, we ensure the temporal errors dominate spatial errors by choosing a modestly large number of spatial elements, so that the effects of changes to the spatial discretization parameters can be observed.  Finally we consider the effect of the spatial discretization in section \ref{sec:example_Par_space_scale}. In all cases the effectivity ratio of the error estimator is 1.00. For these examples $\nu = 4, \mu = 1$ in the function $f(x)$ in equation \eqref{eq:temporal_errors_test_rhs}.

 The examples not only provide convincing evidence of the accuracy of the error estimator,
but illustrate the importance of  identifying and estimating distinct error contributions. We
observe fortuitous cancellation in the first example so that further refinement actually
increases the error. The second example demonstrates a situation in which the refinement
strategy has no effect on the overall error since it does not address the dominant source
of error. Later examples show that when the dominant error term is targeted and it is at
least  an order of magnitude larger than all the other error contributions, refinements strategies
have the anticipated effect on the overall error.

\subsubsection{Effect of the number of Parareal iterations}
\label{sec:example_Par_iterations}

The total error in Table \ref{tab:example_Par_iterations} initially decreases with the number of Parareal iterations ($\NParIte$), but then increases somewhat. The initial decrease is expected, since for $\NParIte=1$, the iteration error $\Kterror$ is the dominant source of error, and hence increasing the number of iterations leads to a decrease in this iteration error, and hence the overall error as well. However, after a single Parareal iteration, the iteration error $\Kterror$ is no longer the dominant error component. Rather the discretization error $\fterror$, which remains essentially constant during the iterative process, becomes the dominant error. Moreover, the discretization and iteration errors have opposite signs, and there is a fortuitous cancellation of error for $\NParIte=2$ between these two terms. When the number of iterations is increased to 3, $\Kterror$ decreases as expected, but so does the cancellation of error between this term and $\fterror$, and hence the total error shows a modest increase.

\begin{table}[H]
\centering
\begin{tabular}{c|c|c|c|c|c|c}
\toprule
$\NParIte$ & Est. Err. & $\gamma$ &  $\fterror$ & $\Kterror$ & $\cterror$ & $\auxerror$  \\
\midrule
1&  -1.02e-01 & 1.00e+00 &  5.10e-02 &  -1.53e-01 & 0.00e+00 &  -2.30e-06\\
2&  4.39e-02 &  1.00e+00 &  5.82e-02 &  -1.43e-02 & 2.86e-06 &  -2.91e-06\\
3&  5.73e-02 &  1.00e+00 &  5.89e-02 &  -1.59e-03 & 3.11e-06 &  -2.96e-06\\
\bottomrule
\end{tabular}
\caption{$\Nctime$ = 20,$\reffactime$ = 16, $\NParSub$ = 10,    $\Ncspace$ = 20, $\Qcspace$ = 1, $\Qfspace$ = 2 }
\label{tab:example_Par_iterations}
\end{table}

\subsubsection{Effect of the number of temporal subdomains}
\label{sec:example_Par_subdomains}

Increasing the number of temporal subdomains ($\NParSub$) is not expected to affect the discretization error $\fterror$ which is determined by the spatial and temporal element scales, but is expected to increase the iteration error $\Kterror$. Both are supported by the results in Table \ref{tab:example_Par_subdomains} below. 
The discretization error remains the largest error for all choices of the number of temporal subdomains.

\begin{table}[H]
\centering
\begin{tabular}{c|c|c|c|c|c|c}
\toprule
$\NParSub$ & Est. Err. & $\gamma$ &  $\fterror$ & $\Kterror$ & $\cterror$ & $\auxerror$  \\
\midrule
2&  1.16e-01 &  1.00e+00 &  1.16e-01 &  3.53e-07 &  -5.81e-09 & -3.58e-08\\
5&  1.16e-01 &  1.00e+00 &  1.16e-01 &  3.68e-05 &  6.18e-09 &  1.02e-07\\
10& 1.12e-01 &  1.00e+00 &  1.15e-01 &  -3.26e-03 & 1.52e-08 &  9.56e-08\\
\bottomrule
\end{tabular}
\caption{$\Nctime$ = 40,$\reffactime$ = 4, $\NParIte$ = 2,    $\Ncspace$ = 20, $\Qcspace$ = 1, $\Qfspace$ = 2 }
\label{tab:example_Par_subdomains}
\end{table}

\subsubsection{Effect of the fine time scale}
\label{sec:example_Par_fine_time_scale}

As is evident in Table \ref{tab:example_Par_fine_time_scale}, increasing the ratio $\reffactime$ between the temporal fine and coarse scales reduces the discretization error $\fterror$. Since this is the dominant error, this also leads to a decrease in the total error.
\begin{table}[H]
\centering
\begin{tabular}{c|c|c|c|c|c|c}
\toprule
\reffactime & Est. Err. & $\gamma$ &  $\fterror$ & $\Kterror$ & $\cterror$ & $\auxerror$  \\
\midrule
2&  7.21e-01 &  1.00e+00 &  7.31e-01 &  -9.49e-03 & 1.78e-04 &  -3.77e-04\\
4&  3.82e-01 &  1.00e+00 &  4.06e-01 &  -2.31e-02 & 2.75e-04 &  -4.08e-04\\
\bottomrule
\end{tabular}
\caption{$\Nctime$ = 10,$\NParSub$ = 10, $\NParIte$ = 2, $\Qctime$ = 1, $\Qftime$ = 1,   $\Ncspace$ = 20, $\Qcspace$ = 1, $\Qfspace$ = 2 }
\label{tab:example_Par_fine_time_scale}
\end{table}

\subsubsection{Effect of the coarse time scale}
\label{sec:example_Par_coarse_time_scale}

The results in Table \ref{tab:example_Par_coarse_time_scale} demonstrate that improving the accuracy of the coarse temporal solution reduces \emph{all} components of the error, provided that the temporal errors dominate the spatial errors.

\begin{table}[H]
\centering
\begin{tabular}{c|c|c|c|c|c|c}
\toprule
$\Nctime$ & Est. Err. & $\gamma$ &  $\fterror$ & $\Kterror$ & $\cterror$ & $\auxerror$  \\
\midrule
10& 7.21e-01 &  1.00e+00 &  7.31e-01 &  -9.49e-03 & 1.78e-04 &  -3.77e-04\\
20& 4.09e-01 &  1.00e+00 &  4.13e-01 &  -3.91e-03 & 1.42e-06 &  -2.94e-06\\
\bottomrule
\end{tabular}
\caption{$\reffactime$ = 2, $\NParSub$ = 10, $\NParIte$ = 2,    $\Ncspace$ = 20, $\Qcspace$ = 1, $\Qfspace$ = 2 }
\label{tab:example_Par_coarse_time_scale}
\end{table}

\subsubsection{Effect of the spatial scale}
\label{sec:example_Par_space_scale}

For the numerical results presented in Table \ref{tab:example_Par_space_scale} we have increased the number of coarse temporal elements to 
$100$, 
$\reffactime$ to $8$, 
and the number of Parareal iterations to $6$, in order to ensure the temporal errors are small compared with the spatial errors. While decreasing the discretization error as anticipated, improving the spatial accuracy is also seen to decrease the coarse temporal and auxiliary errors (though not monotonically) since these both have a spatial component.

\begin{table}[H]
\centering
\begin{tabular}{c|c|c|c|c|c|c}
\toprule
$\Ncspace$ & Est. Err. & $\gamma$ &  $\fterror$ & $\Kterror$ & $\cterror$ & $\auxerror$  \\
\midrule
   5& 6.61e-02 &  1.00e+00 &  6.61e-02 &  1.00e-10 &  1.16e-07 &  9.08e-07\\
  10& 3.40e-02 &  1.00e+00 &  3.40e-02 &  9.78e-11 &  2.20e-07 &  1.92e-06\\
  20& 2.64e-02 &  1.00e+00 &  2.64e-02 &  9.72e-11 &  2.51e-09 &  6.32e-08\\
\bottomrule
\end{tabular}
\caption{$\Nctime$ = 100, $\reffactime$ = 8, $\NParSub$ = 10, $\NParIte$ = 6,    $\Qcspace$ = 1, $\Qfspace$ = 1 }
\label{tab:example_Par_space_scale}
\end{table}



\subsection{Space-Time Parallel Algorithm}
\label{sec:numerical_results_Par_DD}

The following results were obtained through a combination of the Parareal integration in time and additive Schwarz domain decomposition in space. 
We decompose the error $\fterror$ in to its various components as presented in Theorem \ref{thm:par_dd_err}. The effects of varying the fine and coarse time scales are considered in sections \ref{sec:example_ParDD_fine_time_scale} and \ref{sec:example_ParDD_coarse_time_scale} respectively. The number of domain decomposition iterations is varied in section \ref{sec:example_ParDD_DD_iterations}, the number of spatial subdomains in section \ref{sec:example_ParDD_spatial_subdomains}, and the degree of overlap of the spatial subdomains in \ref{sec:example_ParDD_overlap}. In all examples the effectivity ratio of the error estimator is 1.00. For these examples, we set $\nu = 4, \mu = 2$ in the function $f(x)$ defined by equation \eqref{eq:temporal_errors_test_rhs}. The Richardson parameter used in spatial domain decomposition iterations was always set to $\tau=0.4$.

Once again the examples not only provide convincing evidence of the accuracy of the error estimator,
but illustrate the complex interplay that can occur between different contributions to the overall  error.
We frequently observe second order effects of a refinement strategy on error components other than
those the strategy is directly targeting. These second order effects are largely inconsequential
if the contributions to the overall error have widely different magnitude, but may become important
when the error components are similar in scale and particularly when they are opposite in sign.



\subsubsection{Effect of the fine time scale}
\label{sec:example_ParDD_fine_time_scale}

Consistent with the results in \S \ref{sec:example_Par_fine_time_scale}, decreasing the temporal time step decreases the temporal component $\fterrft$ of the discretization error. Notice that for this example the number of spatial elements has been increased so that the temporal discretization error is dominant. All other error contributions in Table \ref{tab:example_ParDD_fine_time_scale} are largely unaffected.


\begin{table}[H]
\centering
\begin{tabular}{c|c|c|c|c|c|c|c|c}
\toprule
\reffactime & Est. Err. & $\gamma$ &  $\fterrft$ & $\fterrNs$ & $\fterrKs$ & $\Kterror$ & $\cterror$ & $\auxerror$  \\
\midrule
2& 2.42e-01 &  1.00e+00 &  2.45e-01 &  6.94e-08 &  1.11e-02 &  7.98e-04 &  1.23e-02 &  -2.58e-02\\
4& 1.55e-01 &  1.00e+00 &  1.48e-01 &  6.93e-08 &  1.37e-02 &  1.70e-03 &  1.95e-02 &  -2.65e-02\\
8& 1.05e-01 &  1.00e+00 &  8.34e-02 &  7.32e-08 &  2.51e-02 &  2.12e-03 &  2.30e-02 &  -2.68e-02\\
\bottomrule
\end{tabular}
\caption{$\Nctime$ = 10,$\NParSub$ = 10, $\NParIte$ = 2,    $\Ncspace$ = 80, $\NDDSub$ = 2, $\NDDIte$=8, $\beta$ = 0.2, $\Qcspace$ = 1, $\Qfspace$ = 2 }
\label{tab:example_ParDD_fine_time_scale}
\end{table}

\subsubsection{Effect of the coarse time scale}
\label{sec:example_ParDD_coarse_time_scale}

Again, consistent with the results in section \ref{sec:example_Par_coarse_time_scale}, all temporal components of the error decrease as the coarse time scale is decreased. The spatial error components in Table \ref{tab:example_ParDD_coarse_time_scale} are seen to be largely  unaffected since the number of spatial elements has been chosen so that temporal errors dominate.

\begin{table}[H]
\centering
\begin{tabular}{c|c|c|c|c|c|c|c|c}
\toprule
$\Nctime$ & Est. Err. & $\gamma$ &  $\fterrft$ & $\fterrNs$ & $\fterrKs$ & $\Kterror$ & $\cterror$ & $\auxerror$  \\
\midrule
10&   2.42e-01 &  1.00e+00 &  2.45e-01 &  6.94e-08 &  1.11e-02 &  7.98e-04 &  1.23e-02 &  -2.58e-02\\
20&   1.55e-01 &  1.00e+00 &  1.41e-01 &  6.55e-08 &  1.46e-02 &  -9.88e-07 & -5.09e-04 & -4.81e-05\\
40&   1.05e-01 &  1.00e+00 &  7.95e-02 &  6.89e-08 &  2.54e-02 &  -4.12e-07 & -8.73e-06 & -3.05e-07\\
\bottomrule
\end{tabular}
\caption{$\reffactime$ = 2, $\NParSub$ = 10, $\NParIte$ = 2,    $\Ncspace$ = 80, $\NDDSub$ = 2, $\NDDIte$=8, $\beta$ = 0.2, $\Qcspace$ = 1, $\Qfspace$ = 2 }
\label{tab:example_ParDD_coarse_time_scale}
\end{table}

\subsubsection{Effect of the number of domain decomposition iterations}
\label{sec:example_ParDD_DD_iterations}

The spatial iteration error $\fterrKs$ decreases with number of domain decomposition iterations, as shown in Table \ref{tab:example_ParDD_DD_iterations}, while the spatial and temporal discretization errors remain essentially constant. There is a second order effect of decreasing the temporal iterative and coarse time scale error contributions since these error contributions also have a spatial component.

\begin{table}[H]
\centering
\begin{tabular}{c|c|c|c|c|c|c|c|c}
\toprule
$\NDDIte$ & Est. Err. & $\gamma$ &  $\fterrft$ & $\fterrNs$ & $\fterrKs$ & $\Kterror$ & $\cterror$ & $\auxerror$  \\
\midrule
2& 6.61e-01 &  1.00e+00 &  2.16e-01 &  1.95e-05 &  4.49e-01 &  -1.10e-04 & -4.94e-03 & -4.52e-05\\
6& 1.88e-01 &  1.00e+00 &  1.45e-01 &  1.71e-05 &  4.40e-02 &  -1.84e-05 & -1.31e-03 & -4.87e-05\\
\bottomrule
\end{tabular}
\caption{$\Nctime$ = 20, $\reffactime$ = 2, $\NParSub$ = 10, $\NParIte$ = 2,   $\Ncspace$ = 20, $\NDDSub$ = 2, $\beta$ = 0.2, $\Qcspace$ = 1, $\Qfspace$ = 2 }
\label{tab:example_ParDD_DD_iterations}
\end{table}

\subsubsection{Effect of the number of spatial subdomains}
\label{sec:example_ParDD_spatial_subdomains}

The spatial iteration error $\fterrKs$ increases with number of spatial subdomains, while the spatial and temporal discretization errors remain essentially constant. A second order effect is again apparent in Table \ref{tab:example_ParDD_spatial_subdomains}, where the temporal iterative error is see to also increase due to its spatial component.

\begin{table}[H]
\centering
\begin{tabular}{c|c|c|c|c|c|c|c|c}
\toprule
$\NDDSub$ & Est. Err. & $\gamma$ &  $\fterrft$ & $\fterrNs$ & $\fterrKs$ & $\Kterror$ & $\cterror$ & $\auxerror$  \\
\midrule
2& 7.17e-01 &  1.00e+00 &  2.25e-01 &  1.02e-06 &  4.96e-01 &  -8.83e-06 & -4.51e-03 & -4.46e-05\\
4& 1.23e+00 &  1.00e+00 &  3.13e-01 &  1.26e-06 &  9.18e-01 &  3.84e-05 &  -5.76e-03 & -4.43e-05\\
\bottomrule
\end{tabular}
\caption{$\Nctime$ = 20, $\reffactime$ = 2, $\NParSub$ = 10, $\NParIte$ = 2,    $\Ncspace$ = 40, $\NDDIte$=2, $\beta$ = 0.1, $\Qcspace$ = 1, $\Qfspace$ = 2 }
\label{tab:example_ParDD_spatial_subdomains}
\end{table}

\subsubsection{Effect of spatial domain overlap}
\label{sec:example_ParDD_overlap}

Increasing the degree of overlap between the spatial domains is expected to decrease the spatial iterative error $\fterrKs$ while leaving the spatial and temporal discretization errors largely unchanged. Slightly different behavior is observed in Table \ref{tab:example_ParDD_overlap} for this example, perhaps because the temporal discretization error is orders of magnitude larger than the spatial discretization error. Never-the-less, the error estimator is accurate and the effectivity ratio is 1.00.

\begin{table}[H]
\centering
\begin{tabular}{c|c|c|c|c|c|c|c|c}
\toprule
$\beta$ & Est. Err. & $\gamma$ &  $\fterrft$ & $\fterrNs$ & $\fterrKs$ & $\Kterror$ & $\cterror$ & $\auxerror$  \\
\midrule
 0.1& 7.17e-01 &  1.00e+00 &  2.25e-01 &  1.02e-06 &  4.96e-01 &  -8.83e-06 & -4.51e-03 & -4.46e-05\\
 0.2& 6.61e-01 &  1.00e+00 &  2.16e-01 &  1.95e-05 &  4.49e-01 &  -1.10e-04 & -4.94e-03 & -4.52e-05\\
\bottomrule
\end{tabular}
\caption{$\Nctime$ = 20, $\reffactime$ = 2, $\NParSub$ = 10, $\NParIte$ = 2,   $\Ncspace$ = 20, $\NDDSub$ = 2, $\NDDIte$=2, $\Qcspace$ = 1, $\Qfspace$ = 2 }
\label{tab:example_ParDD_overlap}
\end{table}

\subsection{Results for a different time integrator for the Time parallel Algorithm}
\label{sec:numerical_results_Par_CG}

We briefly demonstrate the accuracy of the \emph{a posteriori} error estimates when the continuous Galerkin method, cG($\Qftime$) (see section~\ref{sec:var_mthds}), is employed as the time integration scheme in the fine and coarse scale solvers for the TPA. The approximation space for the coarse and fine scales on the the space-time slab $S^p_{n}$ are  $\Vspacetime{n}{\Qctime}{\Qcspace}$ and $\Vspacetime{n}{\Qftime}{\Qfspace}$ respectively. The results of Theorems~\ref{thm:par_err} and \ref{thm:par_dd_err} remain valid; however, the definition of the residuals involved need to be modified to reflect the cG method. The residual for the cG method is given in \eqref{eq:res_cg_dg}. The results are qualitative similar to those in section~\ref{sec:numerical_results_Par}, and hence we present them without any comment. The aim is to show that the results are still accurate, and that the analysis is applicable to a broad class of numerical methods. The results are given in Tables~\ref{tab:CG_example_Par_iterations}, \ref{tab:CG_example_Par_subdomains}, \ref{tab:CG_example_Par_fine_time_scale}, \ref{tab:CG_example_Par_coarse_time_scale} and \ref{tab:CG_example_Par_space_scale}.

\begin{table}[H]
\centering
\begin{tabular}{c|c|c|c|c|c|c}
\toprule
$\NParIte$ & Est. Err. & $\gamma$ &  $\fterror$ & $\Kterror$ & $\cterror$ & $\auxerror$  \\
\midrule
1&  1.02e-01 & 1.00e+00 &  -4.46e-02 &  1.47e-01 &  0.00e+00 & 1.79e-04\\
2& -6.34e-02 & 1.00e+00 &  -3.86e-02 & -2.48e-02 & -1.99e-04 & 1.80e-04\\
3& -3.81e-02 & 1.00e+00 &  -3.96e-02 &  1.50e-03 & -1.67e-04 & 1.80e-04\\
\bottomrule
\end{tabular}
\caption{$\Nctime$ = 10, $\reffactime$ = 4, $\NParSub$ = 10, $\Qctime$ = 1, $\Qftime$ = 1,   $\Ncspace$ = 20, $\Qcspace$ = 1, $\Qfspace$ = 2 }
\label{tab:CG_example_Par_iterations}
\end{table}

\begin{table}[H]
\centering
\begin{tabular}{c|c|c|c|c|c|c}
\toprule
$\NParSub$ & Est. Err. & $\gamma$ &  $\fterror$ & $\Kterror$ & $\cterror$ & $\auxerror$  \\
\midrule
2& -3.95e-02 & 1.00e+00 &  -3.95e-02 & 3.86e-07 &  1.28e-07 &  1.20e-08\\
5& -3.96e-02 & 1.00e+00 &  -3.95e-02 & -2.55e-05 & 6.94e-05 &  -7.27e-05\\
10&   -6.34e-02 & 1.00e+00 &  -3.86e-02 & -2.48e-02 & -1.99e-04 & 1.80e-04\\
\bottomrule
\end{tabular}
\caption{$\Nctime$ = 10, $\reffactime$ = 4, $\NParIte$ = 2, $\Qctime$ = 1, $\Qftime$ = 1,   $\Ncspace$ = 20, $\Qcspace$ = 1, $\Qfspace$ = 2 }
\label{tab:CG_example_Par_subdomains}
\end{table}

\begin{table}[H]
\centering
\begin{tabular}{c|c|c|c|c|c|c}
\toprule
\reffactime & Est. Err. & $\gamma$ &  $\fterror$ & $\Kterror$ & $\cterror$ & $\auxerror$  \\
\midrule
2& -1.71e-01 & 1.00e+00 &  -1.53e-01 & -1.82e-02 & -1.57e-04 & 1.78e-04\\
4& -6.34e-02 & 1.00e+00 &  -3.86e-02 & -2.48e-02 & -1.99e-04 & 1.80e-04\\
\bottomrule
\end{tabular}
\caption{$\Nctime$ = 10, $\NParSub$ = 10, $\NParIte$ = 2, $\Qctime$ = 1, $\Qftime$ = 1,   $\Ncspace$ = 20, $\Qcspace$ = 1, $\Qfspace$ = 2 }
\label{tab:CG_example_Par_fine_time_scale}
\end{table}

\begin{table}[H]
\centering
\begin{tabular}{c|c|c|c|c|c|c}
\toprule
$\Nctime$ & Est. Err. & $\gamma$ &  $\fterror$ & $\Kterror$ & $\cterror$ & $\auxerror$  \\
\midrule
10&   -1.71e-01 & 1.00e+00 &  -1.53e-01 & -1.82e-02 & -1.57e-04 & 1.78e-04\\
20&   -4.10e-02 & 1.00e+00 &  -3.95e-02 & -1.55e-03 & -5.17e-07 & 6.50e-07\\
\bottomrule
\end{tabular}
\caption{$\reffactime$ = 2, $\NParSub$ = 10, $\NParIte$ = 2, $\Qctime$ = 1, $\Qftime$ = 1,   $\Ncspace$ = 20, $\Qcspace$ = 1, $\Qfspace$ = 2 }
\label{tab:CG_example_Par_coarse_time_scale}
\end{table}

\begin{table}[H]
\centering
\begin{tabular}{c|c|c|c|c|c|c}
\toprule
$\Ncspace$ & Est. Err. & $\gamma$ &  $\fterror$ & $\Kterror$ & $\cterror$ & $\auxerror$  \\
\midrule
   5& 3.73e-02 &  1.00e+00 &  3.69e-02 &  1.68e-10 &  -1.95e-05 & 8.85e-04\\
  10& 5.55e-03 &  1.00e+00 &  5.54e-03 &  1.40e-10 &  -7.21e-07 & 4.15e-05\\
  20& -1.93e-03 & 1.00e+00 &  -1.93e-03 & 1.33e-10 &  -6.45e-07 & 2.55e-07\\
\bottomrule
\end{tabular}
\caption{$\Nctime$ = 20, $\reffactime$ = 6, $\NParSub$ = 10, $\NParIte$ = 6, $\Qctime$ = 1, $\Qftime$ = 1,   $\Qcspace$ = 1, $\Qfspace$ = 1 }
\label{tab:CG_example_Par_space_scale}
\end{table}

\section{Conclusions and future work}
\label{sec:conclusions}

We first developed an accurate adjoint-based \emph{a posteriori error} analysis for the Time Parallel Algorithm, which applies the  Parareal  method in time for the solution of parabolic partial differential equations. This analysis  does not seek to separate spatial and temporal sources of error, but combines the two as ``discretization'' error. Additional error contributions arise due to incomplete iteration, discontinuities in the coarse forward solution, and the fine adjoint solution when it is solved in parallel. We then extended this analysis to the Space-Time Parallel Algorithm  by assuming that the spatial solution is determined through a second iterative method, in this case domain decomposition. The combined analysis is able to disaggregate the spatial and temporal discretization errors, as well as  identify additional iterative errors resulting from the domain decomposition iteration in space. Thus the analysis presented here provides a basis for separating discretization and iteration errors and for estimating the effects of incomplete iteration in both space and time. Accurate error estimates provide a foundation for adaptivity and are essential for accurate uncertainty quantification where it is necessary to distinguish variation due to parameter changes from  variation  due to numerical errors which can also vary across the parameter domain. 

We have limited the analysis to linear problems and intend to extend these results  to nonlinear problems using linearization techniques that have previously proven effective~\cite{chaudhry2016posteriori}. We  also investigate more sophisticated temporal solvers than backward Euler and a simple cG method. Parallel iterative methods for solving PDEs require a large number of discretization choices. The error analysis developed here, which accurately distinguishes multiple sources of error, provides a sound foundation on which to make many of these choices. Finally, we can investigate adaptive strategies, noting the complex interaction between error components.

Waveform relaxation methods \cite{gander2007optimized, gander2013parareal, gander1998space} make a fundamentally different choice when combining the two iterative methods of domain decomposition and Parareal iteration. 
Assume  we wish to solve a parabolic partial differential equation on $\Omega \times (0,T]$ and let $\Omega_i \subset \Omega, i=1, \dots p$ be a set of overlapping (spatial) subdomains. Waveform relaxation methods consider domain decomposition as the outer iteration and employs Parareal iteration (or some other time integration technique) to solve subproblems on each spatio-temporal subdomain $\Omega_i \times (0,T], i=1, \dots p$ independently, and then perform a domain decomposition like iteration on the $p$ spatial-temporal blocks. Eficient implementations of waveform relaxation require additional computation to determine Robin conditions for the boundaries of subdomains. An analysis of this approach is saved for future work, starting with an \emph{ a posteriori} error analysis for waveform relaxation implementing a simple, discontinuous Galerkin method for time integration and then extending to Parareal integration in time.

\section*{Acknowledgments}
J. Chaudhry’s work is supported by the NSF-DMS 1720402.
S. Tavener's work is supported by NSF-DMS 1720473.
D. Estep's work is supported by NSF-DMS 1720473 and NSERC grants.

\bibliographystyle{plain}
\bibliography{refs_parareal_pde}
\appendix

\section{Notation}
\label{sec:appendix}

\noindent{\bf Variables, errors and residuals}

\begin{table}[H]
\centering
\begin{tabular}{||l|l||}
\toprule
& Solutions \\
\midrule
$u$                     & Analytic solution \\
$U^p$                   & Fine scale discrete solution on the $p$th temporal subdomain                                 \\
$\widehat{U}^p$         & Coarse scale discrete solution on the $p$th temporal subdomain                               \\
$U^{p,\{k_t\}}$         & Fine scale discrete solution on the $p$th temporal subdomain after $k_t$ parareal iterations \\
$U^{p,\{k_t, k_s\}}$    & Fine scale discrete solution on the $p$th temporal subdomain after $k_t$ parareal iterations \\
                        & \enskip and $k_s$ domain decomposition iterations   \\
\midrule
& Errors and residuals \\
\midrule
$e^p$                      & Fine scale error on the $p$th temporal subdomain                      \\
$\widehat{e}^p$            & Coarse scale error on the $p$th temporal subdomain                    \\
$\mathcal{R}^p$            & Fine scale residual on the $p$th temporal subdomain                   \\
$\widehat{\mathcal{R}}^p$  & Coarse scale residual on the $p$th temporal subdomain                 \\
$\widehat{\mathcal{R}}_{\rm aux}^p$  & Auxilary adjoint problem residual on the $p$th temporal subdomain   \\
\midrule
& Adjoints \\
\midrule
$\phi$              & Fine scale adjoint solution                                   \\
$\phi^p$            & Fine scale adjoint solution on the $p$th temporal subdomain   \\
$\widehat{\phi}$    & Coarse scale adjoint solution                                 \\
$\widehat{\phi}_{\rm aux}^p$ & Auxiliary adjoint solution on the $p$th temporal subdomain    \\
\bottomrule
\end{tabular}
\caption{Notation: Variables, errors and residuals}
\label{tab:notation_variables}
\end{table}

\bigskip
\noindent{\bf Discretization choices}

\begin{table}[H]
\centering
\begin{tabular}{||l|l||}
\toprule
& Time \\
\midrule
$N_t$             & Number of fine scale temporal finite elements                           \\
$\widehat{N}_t$   & Number of coarse scale temporal finite elements                         \\
$\reffactime$     & $N_t/\widehat{N}_t$. Refinement factor for time                         \\
$q_t$             & Interpolation degree of fine scale temporal finite elements             \\
$\widehat{q}_t$   & Interpolation degree of coarse scale temporal finite elements           \\
$P_t$             & Number of parareal subdomains (synchronization intervals)               \\
$K_t$             & Number of parareal iterations                                           \\
\midrule
& Space \\
\midrule
$N_s$             & Number of fine scale spatial finite elements                           \\
$\widehat{N}_s$   & Number of coarse scale spatial finite elements                         \\
$q_s$             & Interpolation degree of fine scale spatial finite elements             \\
$\widehat{q}_s$   & Interpolation degree of coarse scale spatial finite elements           \\
$P_s$             & Number of spatial subdomains                                           \\
$K_s$             & Number of domain decomposition iterations                              \\
$\beta$           & Amount of spatial overlap                                              \\
\bottomrule
\end{tabular}
\caption{Notation: Discretization choices}
\label{tab:notation_discrization_choices}
\end{table}

\bigskip
\noindent{\bf Time discretization}

\begin{table}[H]
\centering
\begin{tabular}{||l|l||}
\toprule
$T_{p-1}$  & Left hand end of pth time subdomain \\
$T_p$      & Right hand end of $p$th time subdomain \\
\midrule
& Course scale solution \\
\midrule
$\widehat{N}^p_t$      & Number of coarse scale temporal finite elements in $p$th time subdomain   \\
$\widehat{t}^p_{j-1}$  & Left hand end of $j$th coarse time interval in $p$th time subdomain       \\
$\widehat{t}^p_j $     & Right hand end of $j$th coarse time interval in $p$th time subdomain      \\
$\widehat{\mathcal{I}}^p_j$ & $[ \widehat{t}^p_{j-1}, \widehat{t}^p_j ]$                           \\
$\widehat{\mathcal{T}}^p$ & $\left \{ \widehat{\mathcal{I}}^p_1, \dots, \widehat{\mathcal{I}}^p_j, \dots, \widehat{\mathcal{I}}^p_{\widehat{N}^p_t} \right\} $                                              \\
\midrule
& Fine scale solution \\
\midrule
$N^p_t$      & Number of fine scale temporal finite elements in $p$th time subdomain     \\
$t^p_{j-1}$  & Left hand end of $j$th fine time interval in $p$th time subdomain         \\
$t^p_{j}$    & Right hand end of $j$th fine time interval in $p$th time subdomain        \\
$\mathcal{I}^p_j$      & $[ t^p_{j-1}, t^p_j ]$                                          \\
$\mathcal{T}^p$ & $\left \{ \mathcal{I}^p_1, \dots, \mathcal{I}^p_j, \dots, \mathcal{I}^p_{N^p_t} \right \}$  \\
\bottomrule
\end{tabular}
\caption{Notation: Time discretization}
\label{tab:notation_time_discretization}
\end{table}

\section{Standard Parareal algorithm and equivalence}
\label{sec:appendix_Parareal}


The standard Parareal algorithm only defines the solutions at the times $T_p$~\cite{Maday08}. This algorithm is given in Algorithm~\ref{alg:parareal_standard}
Here $\widetilde{U}_p^{(k_t)} \in \Vspace{\Qcspace}$ and $ \overline{U}_p^{(k_t)} \in \Vspace{\Qfspace}$ are the coarse and fine scale solutions at $T_{p}$ at iteration $k_t$. 


\begin{algorithm}[H]\caption{Standard form of the Parareal algorithm}
\label{alg:parareal_standard}
\begin{algorithmic}
\Procedure{PAR}{$\NParSub, \NParIte, \csoln_0$}
\State $\corr{p}{0} = 0, p=0,\ldots,\NParSub$ \Comment{Initialize corrections}
\For{$k_t = 1, \ldots, \NParIte$}
  \State $\widetilde{U}_{0}^{(k_t)} := \csoln_0$  \Comment{Set initial conditions}
  \For{$p=1, \ldots, \NParSub$}
    \State $\widetilde{U}_p^{(k_t)} = \widehat{G}^p \left[ \widetilde{U}_{p-1}^{(k_t)} \right] (T_p) + C_p^{k_t-1} $
           \Comment{Serial computation}
    \State $\overline{U}_p^{(k_t)} = F^p \left[\widetilde{U}_{p-1}^{(k_t)} \right] (T_p)$
           \Comment{Parallel computation}
    \State $C_p^{k_t}  = \overline{U}_p^{(k_t)} - \widehat{G}^p \left[\widetilde{U}_{p-1}^{(k_t)} \right] (T_p)$
           \Comment{Update corrections}
  \EndFor
\EndFor
\EndProcedure
\end{algorithmic}
\end{algorithm}

\begin{thm}
The standard Parareal algorithm in Algorithm~\ref{alg:parareal_standard} and the variational version  in Algorithm \ref{alg:parareal_variational} are equivalent in the sense that,
\begin{align}
\widetilde{U}_p^{(\NParIte)}=&\csolnpk{p}{\NParIte}(x,T_p) + \corr{p}{\NParIte-1}\\
\overline{U}_p^{(\NParIte)}=&\fsolnpk{p}{\NParIte}(x,T_p)\\
\corr{p}{\NParIte}=&\corr{p}{\NParIte}
\end{align}
\end{thm}

For proof, please see \cite{chaudhry2016posteriori}.

\section{Coarse scale error}
\label{sec:appendix_coarse_error}

While the focus of the article is on the fine scale error, we briefly outline the result for the coarse scale error for completeness. Let $\cerrpk{p}{k_t} = \soln -  \csolnpk{p}{k_t}$. Then,

\begin{thm}
\begin{eqnarray}
\label{eq:coarse_err}
(\psi_T, \cerrpk{\NParSub}{k_t}(T))   =  \sum_{p=1}^{\NParSub} \crespuphi{p}{\csolnpk{p}{k_t}}{\cadj} - \sum_{p=1}^{\NParSub-1} (\cadj(T_{p}),\corr{p}{k_t-1})
    + (\cadj(0), \cerrpk{1}{k_t}(0)).
\end{eqnarray}
\end{thm}
The proof  is similar to the proof of Theorem \ref{thm:par_err}.

\end{document}